\documentclass[aap]{stylefile}
\usepackage{amsthm,amsmath,amsfonts,amssymb}
\usepackage{natbib}
\usepackage{mathtools}
\usepackage[shortlabels]{enumitem}
\usepackage{algorithm}
\usepackage{algorithmic}
\usepackage{tikz}
\usepackage[colorlinks,citecolor=blue,urlcolor=blue]{hyperref}
\usepackage{graphicx}
\usepackage{ragged2e}
\usepackage{bm}
\usepackage{multicol}
\usepackage{multirow}
\usepackage{pifont}

\startlocaldefs
\theoremstyle{plain}
\newtheorem{theorem}{Theorem}[section]
\newtheorem{lemma}{Lemma}[section]
\newtheorem{corollary}{Corollary}[theorem]
\newtheorem{proposition}{Proposition}[section]

\theoremstyle{definition}
\newtheorem{definition}{Definition}[section]
\newtheorem{condition}{Condition}[section]

\newtheorem{assumption}{Assumption}[section]

\theoremstyle{remark}
\newtheorem*{remark}{Remark}

\endlocaldefs

\begin{document}

\begin{frontmatter}
\title{Finite-Sample Analysis of Nonlinear Stochastic Approximation with Applications in Reinforcement Learning}

\begin{aug}
\author[A]{\fnms{Zaiwei} \snm{Chen}\ead[label=e1,mark]{zchen458@gatech.edu}},
\author[A]{\fnms{Sheng} \snm{Zhang}\ead[label=e2,mark]{shengzhang@gatech.edu}},
\author[B]{\fnms{Thinh T.} \snm{Doan}\ead[label=e4]{thinhdoan@vt.edu}},
\author[C]{\fnms{John-Paul} \snm{Clarke}\ead[label=e5]{johnpaul@utexas.edu}},\\
\and
\author[A]{\fnms{Siva Theja} \snm{Maguluri}\ead[label=e3,mark]{siva.theja@gatech.edu}}

\address[A]{
Geogia Institute of Technology,
\printead{e1,e2,e3}}

\address[B]{
Virginia Tech,
\printead{e4}}

\address[C]{
The University of Texas at Austin,
\printead{e5}}

\end{aug}

\begin{abstract}
Motivated by applications in reinforcement learning (RL), we study a nonlinear stochastic approximation (SA) algorithm under Markovian noise, and establish its finite-sample convergence bounds under various stepsizes. Specifically, we show that when using constant stepsize (i.e., $\alpha_k\equiv \alpha$), the algorithm achieves exponential fast convergence to a neighborhood (with radius $O(\alpha\log(1/\alpha))$) around the desired limit point. When using diminishing stepsizes with appropriate decay rate, the algorithm converges with rate $O(\log(k)/k)$. Our proof is based on Lyapunov drift arguments, and to handle the Markovian noise, we exploit the fast mixing of the underlying Markov chain. 

To demonstrate the generality of our theoretical results on Markovian SA, we use it to derive the finite-sample bounds of the popular $Q$-learning with linear function approximation algorithm, under a condition on the behavior policy. Importantly, we do not need to make the assumption that the samples are i.i.d., and do not require an artificial projection step in the algorithm to maintain the boundedness of the iterates. Numerical simulations corroborate our theoretical results.
\end{abstract}
\end{frontmatter}

\section{Introduction}
Large-scale optimization and machine learning problems are typically solved using stochastic approximation (SA) methods (i.e., iterative algorithms in the presence of noise). For example, in optimization, the stochastic gradient descent (SGD) algorithm is commonly used to find an optimal solution of a target objective function \citep{bottou2018optimization}. In reinforcement learning (RL), $Q$-learning and TD-learning are popular algorithms used to solve the Bellman equations \citep{bertsekas1996neuro}. 

The behavior of SA algorithms is highly dependent on the nature of the associated noise (e.g., i.i.d., martingale difference, or Markovian). In robust optimization problems as considered in \cite{duchi2012ergodic} where the data is generated by an auto-regressive process, the corresponding SGD algorithms naturally involve Markovian noise. In RL, algorithms such as $Q$-learning and TD-learning use sample trajectories from a Markov decision process (MDP) to carry out the update, and hence can also be modeled using Markovian SA algorithms.

The asymptotic convergence of SA algorithms with Markovian noise has been studied extensively in the literature \citep{benveniste2012adaptive,borkar2009stochastic,bertsekas1996neuro}. Beyond asymptotic convergence, it is of more practical interest to study finite-sample guarantees, i.e., to provide performance guarantees on the output of SA algorithms after performing a finite number of iterations. More formally, suppose we perform $k$ iterations of an SA algorithm and denote the output by $\theta_k$. Then the goal is to understand how the quantity $\mathbb{E}[\|\theta_k-\theta^*\|^2]$ decay as a function of $k$, where $\theta^*$ is the desired limit point, and $\|\cdot\|$ is a suitable norm. This leads to our main contributions in the following.

\subsection{Main Contributions}
The major contributions of this paper are twofold.

\textbf{Finite-Sample Analysis for Nonlinear Markovian SA.} We establish finite-sample convergence guarantees for nonlinear SA with Markovian noise for using various stepsizes, where we do not require an artificial projection step in the algorithm. The results state that Markovian SA algorithms enjoy exponential convergence rate to a neighborhood around the desired limit when using constant stepsize, and $O(\log(k)/k)$ convergence rate when using appropriate diminishing stepsizes. We prove the results by applying a suitable Lyapunov function on the stochastic iterates, and show that in expectation it produces a negative drift. To handle the Markovian noise, we exploit the geometric mixing of the underlying Markov chain.

\textbf{Finite-Sample Analysis of $Q$-Learning with Linear Function Approximation.} To demonstrate the effectiveness of our SA results, we use them to establish for the first time finite-sample bounds for the $Q$-learning with linear function approximation algorithm. Since the algorithm does not necessarily converge \citep{baird1995residual}, we use our SA results to provide a sufficient condition under which $Q$-learning with linear function approximation converges. In addition, we verify the sufficiency of our proposed condition and the resulting convergence rates via numerical experiments based on a well-known divergent counter-example of $Q$-learning from \cite{baird1995residual}. Specifically, we demonstrate that if our condition is satisfied, the algorithm converges, and the rates match with our theoretical results. 

\subsection{Related Literature}
We will first present related work on SA, then on $Q$-learning, and $Q$-learning with linear function approximation.

\textbf{SA Method.} The SA method, originally proposed in \cite{robbins1951stochastic}, is an iterative method for solving root-finding problems with incomplete imformation. The asymptotic behavior of SA algorithms is captured by its associated ordinary differential equation (ODE), which leads to the popular ODE approach for analyzing SA algorithms \citep{benveniste2012adaptive,kushner2012stochastic}. Specifically, given certain assumptions, it was shown in \cite{ljung1977analysis,borkar2009stochastic} that the SA algorithm converges almost surely as long as the corresponding ODE is stable. The ODE approach was extended to more general cases in \cite{benaim1996dynamical,yaji2019analysis,karmakar2021stochastic}, where the ODE lacks stability, or has multiple equilibrium points. The convergence of various SA algorithms such as SA with Markovian noise and multiple time-scale SA was studied in  \cite{ramaswamy2018stability,karmakar2021stochastic} and \cite{bhatnagar1998two,bhatnagar1997multiscale} respectively. While the results presented were very general, they study SA algorithms in the asymptotic regime. In this paper, we perform finite-sample analysis, which is different in flavor and provides stronger finite-sample convergence guarantees.
	
For linear SA algorithms, finite-sample mean-square bounds were established for both i.i.d. sampling and Markovian sampling in \cite{bhandari2018finite,srikant2019finite}. Concentration results were established in \cite{dalal2018finite,thoppe2019concentration}. For non-linear SA algorithms, finite-sample bounds in general are only derived in a special form of SA, namely SGD \cite{bottou2018optimization,lan2020first,moulines2011non}. Moreover, unlike i.i.d. sampling, in the case of Markovian sampling, an artificial projection (onto a ball) is introduced in the algorithm to ensure that the iterates are bounded \cite{duchi2012ergodic}. 

\textbf{$Q$-Learning (with Linear Function Approximation).}  $Q$-learning \cite{watkins1992q} is perhaps one of the most popular algorithms for solving RL problems \cite{bertsekas1996neuro,sutton2018reinforcement}. The asymptotic convergence and finite-sample guarantees of $Q$-learning were studied in \cite{tsitsiklis1994asynchronous,jaakkola1994convergence,borkar2000ode} and \cite{even2003learning,beck2012error,kearns1998finite}, respectively.

A major limitation with $Q$-learning is that it becomes computationally intractable when the size of the state-action space is large. One way to overcome this difficulty is to use function approximation. In this work, we consider the $Q$-learning with linear function approximation algorithm, which can be modeled as a nonlinear Markovian SA algorithm \cite{melo2008analysis}. However, as shown by the counter-example in \cite{baird1995residual}, $Q$-learning with linear function approximation does not necessarily converge, which was identified as a major theoretical open problem in RL \cite{sutton1999open}. Its asymptotic convergence was established in \cite{melo2008analysis} under an additional condition on the behavior policy. Under a similar condition, we establish its finite-sample bounds by exploiting some natural properties of $Q$-learning (such as Lipschitz continuity), and the fast mixing of finite-state Markov chains. The mixing time argument for dealing with Markovian noise was inspired by \cite[Section 4.4]{bertsekas1996neuro} and \cite{srikant2019finite}, where linear SA was studied. Importantly, our approach does not require a projection step in the algorithm \cite{bhandari2018finite}, which is impractical in RL since one needs to know the problem parameters to pick the projection set so that the desired limiting solution lies in it.

\section{Nonlinear SA with Markovian Noise}\label{sec:sa}
Consider the problem of solving for $\theta^*$ in the equation
\begin{align}\label{eq:sa}
	\bar{F}(\theta)=\mathbb{E}_{\mu_X}[F(X,\theta)]=0,
\end{align}
where $X\in\mathcal{X}\subseteq\mathbb{R}^{n_X}$ is a random vector with distribution $\mu_X$, and the function $F:\mathcal{X}\times\mathbb{R}^d\mapsto\mathbb{R}^d$ is a general nonlinear operator. When the distribution $\mu_X$ is unknown, Eq. (\ref{eq:sa}) cannot be solved analytically. Therefore, we consider solving the equation using the SA method. With initialization $\theta_0\in\mathbb{R}^d$, the estimate $\theta_k$ of $\theta^*$ is updated according to
\begin{align}\label{sa:algorithm}
	\theta_{k+1}=\theta_k+\alpha_k (F(X_k,\theta_k)+w_k),
\end{align}
where $\{X_k\}$ is a uniformly ergodic Markov chain with stationary distribution $\mu_X$, $\{w_k\}$ represents the additive martingale difference noise that possibly depends on $\{\theta_k\}$, and $\{\alpha_k\}$ is the stepsize sequence. To better understand Algorithm (\ref{sa:algorithm}), consider the special case where $F(x,\theta)=-\nabla J(\theta)+x$ for some cost function $J(\cdot)$, Algorithm (\ref{sa:algorithm}) is the popular SGD algorithm for minimizing $J(\cdot)$.

The behavior of SA algorithm (\ref{sa:algorithm}) is closely related to the trajectory of the ODE
\begin{align}\label{sa:ode}
	\dot{\theta}(t)=\bar{F}(\theta(t)).
\end{align}
A popular approach to analyze an ODE is to construct a Lyapunov function and study the time-derivative of the Lyapunov function along the trajectory of the ODE. Inspired by the Lyapunov technique for the ODE stability analysis, in this paper, we directly study SA algorithm (\ref{sa:algorithm}) using a Lyapunov approach. See \cite{fazlyab2017dynamical,hu2019characterizing,hu2017unified,franca2018admm,romero2020finite} for more details on using Lyapunov functions to study the behavior of iterative algorithms. Since Algorithm (\ref{sa:algorithm}) is a \textit{discrete} and \textit{stochastic} counterpart of ODE (\ref{sa:ode}), a major challenge is to handle the error caused by the discretization and the noise. We begin by stating our assumptions to study SA algorithm (\ref{sa:algorithm}). Let $\|\cdot\|$ be the $\ell_2$-norm for vectors and the induced $2$-norm for matrices. 
\begin{assumption}\label{as:Lipschitz}
	There exists constant $L_1>0$ such that (1) $\|F(x,\theta_1)-F(x,\theta_2)\|\leq L_1\|\theta_1-\theta_2\|$ for all $\theta_1$, $\theta_2$, and $x$, and (2) $\|F(x,0)\|\leq L_1$ for all $x$.
\end{assumption}

Assumption \ref{as:Lipschitz} states that the operator $F(x,\theta)$ is $L_1$-Lipschitz continuous with respect to $\theta$ uniformly in $x$. In the special case where $F(x,\theta)$ is a linear function of $\theta$ as considered in \cite{bhandari2018finite,srikant2019finite}, i.e., $F(x,\theta)=A(x)\theta+b(x)$, Assumption \ref{as:Lipschitz} is satisfied when $\sup_{x\in\mathcal{X}}\|A(x)\|< \infty$ and $\sup_{x\in\mathcal{X}}\|b(x)\|< \infty$. In our setting, although $F(x,\theta)$ is a nonlinear function of $\theta$, Assumption \ref{as:Lipschitz} implies that the growth rate of both $\|F(x,\theta)\|$ and $\|\bar{F}(\theta)\|$ can at most be affine in terms of $\|\theta\|$. To see this, under Assumption \ref{as:Lipschitz}, we have by the triangle inequality and Jensen's inequality that 
\begin{align}
	\|F(x,\theta)\|&\leq L_1\|\theta\|+\|F(x,0)\|\leq L_1(\|\theta\|+1),\label{affine_growth_rate-F}\\
	\|\bar{F}(\theta)\|&\leq \mathbb{E}_{\mu_X}[\|F(X,\theta)\|]\leq L_1(\|\theta\|+1).\label{affine_growth_rate-barF}
\end{align}
These properties for $F(x,\theta)$ and $\bar{F}(\theta)$ essentially let us establish the finite-sample bounds akin to the case where $F(x,\theta)$ is a linear function of $\theta$.

\begin{assumption}\label{as:stability}
	The target equation $\bar{F}(\theta)=0$ has a unique solution, which we have denoted by $\theta^*$, and there exists $c_0>0$ s.t. $(\theta-\theta^*)^\top\bar{F}(\theta)\leq -c_0\|\theta-\theta^*\|^2$ for all $\theta\in\mathbb{R}^d$.
\end{assumption}

In the SGD setting (i.e., $F(x,\theta)=-\nabla J(\theta)+x$), Assumption \ref{as:stability} is satisfied when the objective function $J(\cdot)$ is strongly convex. Moreover, Assumption \ref{as:stability} can be viewed as an exponential dissipativeness property of the ODE (\ref{sa:ode}) with a quadratic storage function. In fact, this assumption guarantees that $\theta^*$ is the unique exponentially stable equilibrium point of ODE (\ref{sa:ode}). To see this, let $W(\theta)=\|\theta-\theta^*\|^2$ be a candidate Lyapunov function. Then we have by Assumption \ref{as:stability} that 
\begin{align}\label{Vdot}
	\frac{d}{dt}W(\theta(t))=2(\theta(t)-\theta^*)^\top \dot{\theta}(t)\leq -2c_0 W(\theta(t)),
\end{align}
which implies that $W(\theta(t))\leq W(\theta(0))e^{-2c_0 t}$ for all $t\geq 0$. The parameter $c_0$ is called the \textit{negative drift}, and we see that the larger $c_0$ is, the faster $\theta(t)$ converges.

Our next assumption is about the noise sequences $\{X_k\}$ and $\{w_k\}$. Let $\mathcal{F}_k$ be the $\sigma$-algebra generated by $\{\theta_i,X_i,w_i\}_{0\leq i\leq k-1}\cup\{\theta_k,X_k\}$, and let $\|\cdot\|_{\text{TV}}$ be the total variation distance \cite{levin2017markov}. 

\begin{assumption}\label{as:markov-chain}
	(1) The Markov chain $\{X_k\}$ is uniformly geometrically ergodic with unique stationary distribution $\mu_X$. (2) The random process $\{w_k\}$ satisfies $\mathbb{E}[w_k\mid\mathcal{F}_k]=0$ and $\|w_k\|\leq L_2 (\|\theta_k\|+1)$ for all $k\geq 0$, where $L_2>0$ is a constant. 
\end{assumption}

Assumption \ref{as:markov-chain} (1) is made to control the Markovian noise in Algorithm (\ref{sa:algorithm}), and implies that there exist $C\geq 1$ and $\rho\in (0,1)$ such that $\sup_{x\in\mathcal{X}}\|P^k(x,\cdot)-\mu_X(\cdot)\|_{\text{TV}}\leq C\rho^k$ for all $k\geq 0$. When compared to $\{X_k\}$ being i.i.d., the major difference for $\{X_k\}$ being Markovian is that there is a bias in the update, i.e., $\mathbb{E}[F(X_k,\theta)\mid X_0=x]\neq \bar{F}(\theta)$. Since Assumption \ref{as:markov-chain} (1) states that the Markov chain $\{X_k\}$ mixes geometrically fast, it enables us to control such bias and to show that it is not strong enough to alter the desired direction of the update. In the special case where the state-space $\mathcal{X}$ of the Markov chain $\{X_k\}$ is finite, Assumption \ref{as:markov-chain} (1) is satisfied when the Markov chain $\{X_k\}$ is irreducible and aperiodic \cite[Theorem 4.9]{levin2017markov}. Assumption \ref{as:markov-chain} (2) states that $\{w_k\}$ is a martingale difference sequence, and $\|w_k\|$ is allowed to scale affinely with respect to $\|\theta_k\|$.

In addition to these assumptions, the choice of the stepsize sequence $\{\alpha_k\}$ is important. In order to state certain conditions on the stepsizes we pick, we need to use the mixing time of the Markov chain $\{X_k\}$ defined in the following.

\begin{definition}\label{def:mixing_time}
	For any $\delta>0$, the mixing time of the Markov chain $\{X_k\}$ with precision $\delta$ is defined by $t_\delta=\min\{k\geq 0:\sup_{x\in\mathcal{X}}\|P^k(x,\cdot)-\mu_X(\cdot)\|_{\text{TV}}\leq \delta\}$.
\end{definition}

Under Assumption \ref{as:markov-chain} (1), we have for any $\delta>0$ that
\begin{align}
	t_\delta\leq \frac{\log(1/\delta)+\log(C/\rho)}{\log(1/\rho)}\leq L_3(\log(1/\delta)+1),\label{eq:mixing_time_bound}
\end{align}
where $L_3=\frac{\log(C/\rho)}{\log(1/\rho)}$. As a result, we have $\lim_{\delta\rightarrow 0}\delta t_\delta=0$. In fact, we only require $t_\delta=o(1/\delta)$ to carry out our finite-sample analysis. We assume the stronger geometric mixing property merely for an ease of exposition. We next use $t_\delta$ to state our condition on the stepsize sequence $\{\alpha_k\}$. For simplicity of notation, denote $t_k=t_{\alpha_k}$ and $\alpha_{i,j}=\sum_{k=i}^{j}\alpha_k$. Let $L=L_1+L_2$, and assume without loss of generality that $L\geq 1$.
\begin{condition}\label{as:stepsize}
	The stepsize sequence $\{\alpha_k\}$ satisfies the following conditions: (1) $\{\alpha_k\}$ is non-increasing and $\alpha_0\in (0,1)$, and (2) it holds that $\alpha_{k-t_k,k-1}<\frac{c_0}{130L^2}$ for all $k\geq t_k$.
\end{condition}

The reason we impose Condition \ref{as:stepsize} on the stepsize sequence is the following. Recall that a key step in deriving the convergence rate of ODE (\ref{sa:ode}) is to establish the negative drift (cf. Eq. (\ref{Vdot})). Similarly, when deriving finite-sample bounds for the SA algorithm (\ref{sa:algorithm}), there will also be a negative drift term. In addition, there are error terms that arise because of the discretization and the stochastic noise. Using small stepsize helps suppressing these error terms and hence ensures that the negative drift is the dominant term in our analysis.

Suppose we use constant stepsize, i.e., $\alpha_k=\alpha$ for all $k\geq 0$. Since in this case we have $\alpha_{k-t_k,k-1}=\alpha t_\alpha$, and $\lim_{\alpha\rightarrow 0}\alpha t_\alpha=0$, Condition \ref{as:stepsize} is satisfied when $\alpha$ is small enough. In addition to constant stepsize, consider using polynomially diminishing stepsizes of the form $\alpha_k=\alpha/(k+h)^\xi$. We show in Section \ref{pf:thm:diminishing_step_size} that Condition \ref{as:stepsize} is satisfied for any $\alpha>0$ and $\xi\in (0,1]$, provided that $h$ is appropriately chosen. 

\subsection{Finite-Sample Bounds for Nonlinear SA}\label{subsec:sa:finite-sample bounds}
In this section, we present our main results. We begin with the finite-sample bounds of Algorithm (\ref{sa:algorithm}), whose proof is presented in Section \ref{subsec:sa:theorem-proof}.  
\begin{theorem}\label{thm:main}
Consider $\{\theta_k\}$ of Algorithm (\ref{sa:algorithm}). Suppose that Assumptions \ref{as:Lipschitz} -- \ref{as:markov-chain} are satisfied, and $\{\alpha_k\}$ satisfies Condition \ref{as:stepsize}. Let $K=\min\{k:k\geq  t_k\}$. Then we have for all $k\geq K$: 
\begin{align}\label{eq:sa-bound}
	\mathbb{E}[\|\theta_{k}-\theta^*\|^2]
	\leq \beta_1\prod_{j=K}^{k-1}(1-c_0\alpha_j)+\beta_2\sum_{i=K}^{k-1}\hat{\alpha}_i\prod_{j=i+1}^{k-1}(1-c_0\alpha_j),
\end{align}
where $\beta_1=(\|\theta_0\|+\|\theta_0-\theta^*\|+1)^2$, $\beta_2=130L^2(\|\theta^*\|+1)^2$, and $\hat{\alpha}_i=\alpha_i\alpha_{i-t_i,i-1}$.
\end{theorem}	
\begin{remark}
Although the parameter $K$ is defined as $K=\min\{k:k\geq  t_k\}$, we indeed have $K=t_K$. To see this, suppose that $K>t_K$. Since both $K$ and $t_K$ are integers, we must have $K-1\geq t_K\geq t_{K-1}$, where the second inequality follows from the fact that $t_k=t_{\alpha_k}$ is an increasing function of $k$. This contradict to the definition of $K$ and hence we have $K=t_K$.
\end{remark}

On the RHS of Eq. (\ref{eq:sa-bound}), the first term represents the bias due to the initial guess $\theta_0$, and the second term captures the variance due to the noise. Theorem \ref{thm:main} is one of our main contributions in that (1) the function $F(x,\theta)$ is allowed to be nonlinear, (2) it holds when $\{X_k\}$ is a Markov chain instead of being i.i.d., and (3) no modification on Algorithm (\ref{sa:algorithm}) (e.g., adding a projection step) is needed to establish the results.

After establishing the finite-sample bounds of Algorithm (\ref{sa:algorithm}) in its general form, we next consider several common choices of stepsizes, and determine the corresponding convergence rates. We begin by presenting the result when using constant stepsize. The proof of the following corollary is presented in Section \ref{pf:thm:constant_step_size}.

\begin{corollary}\label{thm:constant_step_size}
When $\alpha_k\equiv\alpha$ with $\alpha$ chosen s.t. $\alpha t_\alpha\leq \frac{c_0}{130L^2}$, we have 
\begin{align*}
	\mathbb{E}[\|\theta_{k}-\theta^*\|^2]
	\leq \beta_1(1-c_0\alpha)^{k-t_\alpha}+\beta_2\frac{\alpha t_\alpha}{c_0},\quad \forall\;k\geq t_\alpha.
\end{align*} 
\end{corollary}

We see from Corollary \ref{thm:constant_step_size} that when using constant stepsize, the bias term converges to zero geometrically fast as the number of iterations increases, while the variance term remains as a constant of size $O(\alpha\log(1/\alpha))$. Observe that $t_\alpha\leq L_3(\log(1/\alpha)+1)$ (cf. Eq. (\ref{eq:mixing_time_bound})). Therefore, constant stepsize efficiently eliminates the bias. However, since the noise is added to the iterates without being progressively suppressed, the variance does not converge to zero as $k\rightarrow\infty$.

We next consider diminishing stepsizes. Let \begin{align}\label{eq:stepsize_diminishing}
	\alpha_k=\frac{\alpha}{(k+h)^\xi},
\end{align}
where $\alpha>0$, $\xi\in (0,1]$, and $h$ is chosen such that Condition \ref{as:stepsize} is satisfied. The requirement for choosing $h$ and the proof of the following corollary are presented in Section \ref{pf:thm:diminishing_step_size}.

\begin{corollary}\label{thm:diminishing_step_size}
Suppose $\{\alpha_k\}$ is chosen as in Eq. (\ref{eq:stepsize_diminishing}), then we have the following finite-sample bounds.
\begin{enumerate}[(1)]
	\item \begin{enumerate}[(a)]
		\item When $\xi=1$ and $\alpha<1/c_0$, we have for all $k\geq K$:
	\begin{align*}
		\mathbb{E}[\|\theta_k-\theta^*\|^2]\leq \beta_1\left(\frac{K+h}{k+h}\right)^{c_0\alpha}+\frac{8\beta_2\alpha^2L_3}{1-c_0\alpha}\frac{[\log\left(\frac{k+h}{\alpha}\right)+1]}{(k+h)^{c_0\alpha}}.
	\end{align*}
	\item When $\xi=1$ and $\alpha=1/c_0$, we have for all $k\geq K$:
	\begin{align*}
		\mathbb{E}[\|\theta_k-\theta^*\|^2]\leq\beta_1\left(\frac{K+h}{k+h}\right)+8\beta_2\alpha^2L_3\frac{\log(k+h)[\log\left(\frac{k+h}{\alpha}\right)+1]}{k+h}.
	\end{align*}
	\item When $\xi=1$ and $\alpha>1/c_0$, we have for all $k\geq K$:
	\begin{align*}
		\mathbb{E}[\|\theta_k-\theta^*\|^2]\leq\beta_1\left(\frac{K+h}{k+h}\right)^{c_0\alpha}+\frac{8e\beta_2\alpha^2L_3}{c_0\alpha-1} \frac{\left[\log\left(\frac{k+h}{\alpha}\right)+1\right]}{k+h}.
	\end{align*}
	\end{enumerate}
\item When $\xi\in (0,1)$ and $\alpha>0$, assume without loss of generality that $K\geq [2\xi/(c_0\alpha)]^{1/(1-\xi)}$, then we have for all $k\geq K$:
\begin{align*}
	\mathbb{E}[\|\theta_k-\theta^*\|^2]
	\leq\; \beta_1e^{-\frac{c_0\alpha}{1-\xi}\left((k+h)^{1-\xi}-(K+h)^{1-\xi}\right)}+\frac{4\beta_2\alpha L_3}{c_0}\frac{[\log\left(\frac{k+h}{\alpha}\right)+1]}{(k+h)^\xi}.
\end{align*}
\end{enumerate}
\end{corollary}

Observe from Corollary \ref{thm:diminishing_step_size} (1) that when using $\alpha_k=\alpha/(k+h)$, the constant $\alpha$ must be chosen carefully (i.e., $\alpha>1/c_0$) to achieve the optimal $O(\log(k)/k)$ convergence rate, otherwise the convergence rate is $O(\log(k)/k^{c_0\alpha})$, which can be arbitrarily slow. From Corollary \ref{thm:diminishing_step_size} (2), we see that when $\xi\in (0,1)$, the convergence rate is $O(\log(k)/k^\xi)$, which is sub-optimal, but more robust in the sense that it is independent of $\alpha$. The above analysis indicates that our choice of stepsizes should depend on how precise our estimate of the negative drift parameter $c_0$ is. When our estimate of $c_0$ is accurate, we should use $\alpha_k=\alpha/(k+h)$ with $\alpha>1/c_0$ so that the convergence rate is the optimal $O(\log(k)/k)$. When our understanding to the system model is poor (therefore inaccurate estimate of $c_0$), we should use $\alpha_k=\alpha/(k+h)^\xi$ with $\xi$ chosen close to unity. In that case, we sacrifice the convergence rate for robustness.

Unlike almost sure convergence, where the usual requirement for stepsizes are $\sum_{k=0}^{\infty}\alpha_k=\infty$ and $\sum_{k=0}^{\infty}\alpha_k^2<\infty$ (which corresponds to $\xi\in (1/2,1]$ in our case), we have convergence in the mean-square sense for all $\xi\in (0,1]$. The same phenomenon has been observed in \cite{bhandari2018finite}, where they study linear SA and nonlinear SA with martingale difference noise.

\subsection{Proof of Theorem \ref{thm:main}}\label{subsec:sa:theorem-proof}
In this section, we present the proof of Theorem \ref{thm:main}. Before going into the details, we first provide some intuition. Recall that the Lyapunov function $W(\theta)=\|\theta-\theta^*\|^2$ can be used to show the stability of ODE (\ref{sa:ode}). To analyze the convergence rate of the iterates $\{\theta_{k}\}$ generated by Algorithm (\ref{sa:algorithm}), naturally we want to use the Lyapunov  function $W(\cdot)$ on $\{\theta_k\}$ to show something like
\begin{align}\label{eq:hope-to-show}
	\mathbb{E}[W(\theta_{k +1})]-\mathbb{E}[W(\theta_k)]
	\leq (-c_0\alpha_k + e_1)\mathbb{E}[W(\theta_k)]+e_2.
\end{align}
This is a discrete analog of Eq. (\ref{Vdot}), and so $W(\cdot)$ is a Lyapunov function \cite{haddad2011nonlinear}. In continuous time, Eq. (\ref{Vdot}) enables one to determine the rate of convergence of ODE (\ref{sa:ode}). Eq. \eqref{eq:hope-to-show} is the discrete-time equivalent for SA algorithm (\ref{sa:algorithm}). To make connection to standard control literature, suppose we view $e_2$ as the input. Then when $e_2=0$, Eq. (\ref{eq:hope-to-show}) is of the desired form used to prove asymptotic stability \cite{sontag2008input}. In our case, due to a non-vanishing $e_2$, when using constant stepsize we do not have asymptotic convergence but have convergence to a neighborhood around $\theta^*$. 

Here on the RHS of Eq. (\ref{eq:hope-to-show}), the $-c_0\alpha_k$ term corresponds to the negative drift of the ODE, and the two terms $e_1$ and $e_2$ account for the discretization error and the stochastic error in Algorithm (\ref{sa:algorithm}). The discretization error can be handled using the properties of the function $F(x,\theta)$ (cf. Assumption \ref{as:Lipschitz}) and properly chosen stepsizes (cf. Condition \ref{as:stepsize}). As for the stochastic error, since Markovian noise naturally produces bias in the update, we show that $\mathbb{E}[F(X_k,\theta)\mid X_0=x]$ converges to $\bar{F}(\theta)$ (as $k$ increases) fast enough for any $\theta$, where we make use of Assumption \ref{as:markov-chain} (1). Once we show that both error terms are dominated by the drift term, i.e., $e_1=o(\alpha_k)$ and $e_2=o(\alpha_k)$, Eq. (\ref{eq:hope-to-show}) can be repeatedly used to establish a finite-sample bound of Algorithm (\ref{sa:algorithm}).

Following from the high level idea stated above, we now prove Theorem \ref{thm:main}. To begin, we apply $W(\theta)=\|\theta-\theta^*\|^2$ on the iterates $\theta_k$ of Algorithm (\ref{sa:algorithm}). To utilize the mixing time of the Markov chain $\{X_k\}$, we take expectation conditioning on $X_{k-t_k}$ and $\theta_{k-t_k}$. For simplicity of notation, we use $\mathbb{E}_k[\,\cdot\,]$ for $\mathbb{E}[\cdot\mid X_{k-t_k},\theta_{k-t_k}]$ in the following. Then we have for all $k\geq t_k$: 
\begin{align}
	\mathbb{E}_k[\|\theta_{k+1}-\theta^*\|^2]-\mathbb{E}_k[\|\theta_{k}-\theta^*\|^2]
	=\;&2\mathbb{E}_k[(\theta_k-\theta^*)^\top(\theta_{k+1}-\theta_k)]+\mathbb{E}_k[\|\theta_{k+1}-\theta_k\|^2]\nonumber\\
	=\;&\underbrace{2\alpha_k \mathbb{E}_k[(\theta_k-\theta^*)^\top\bar{F}(\theta_k)]}_{(a)}+\underbrace{2\alpha_k \mathbb{E}_k[(\theta_k-\theta^*)^\top w_k]}_{(b)}\nonumber\\
	&+\underbrace{2\alpha_k \mathbb{E}_k[(\theta_k-\theta^*)^\top(F(X_k,\theta_k)-\bar{F}(\theta_k))]}_{(c)}\nonumber\\
	&+\underbrace{\alpha_k^2\mathbb{E}_k[\|F(X_k,\theta_k)+w_k\|^2]}_{(d)},\label{eq:expand}
\end{align}
where the last line follows by using the update equation (\ref{sa:algorithm}) and by adding and subtracting $\bar{F}(\theta_k)$.

The term $(a)$ corresponds to the negative drift of ODE (\ref{sa:ode}), and we have $(a)\leq -2c_0\alpha_k\mathbb{E}_k\left[\|\theta_k-\theta^*\|^2\right]$ under Assumption \ref{as:stability}.
The term $(b)$ corresponds to the error due to martingale difference noise $\{w_k\}$. Using the tower property of conditional expectation and the assumption that $w_k$ is a martingale difference sequence, we have $(b)=0$. The term $(c)$ corresponds to the error due to the Markovian noise $\{X_k\}$, and the term $(d)$ arises mainly because of the error due to discretization. What remains to be shown is that the terms $(c)$ and $(d)$ are dominated by the term $(a)$. We begin by bounding the term $(d)$ in the following lemma, whose proof is presented in Section \ref{pf:le_1}.
\begin{lemma}\label{le_1}
The following inequality holds for all $k\geq t_k$:
\begin{align*}
(d)\leq 2L^2\alpha_k^2\left[\mathbb{E}_k[\|\theta_k-\theta^*\|^2]+(\|\theta^*\|+1)^2\right].
\end{align*}
\end{lemma}

Observe that Lemma \ref{le_1} implies that $(d)=O(\alpha_k^2)=o(\alpha_k)$, which is what we desire. We next consider the term $(c)$. To control it, we need the following two results. 
\begin{lemma}\label{le:mixing_time}
For any given $\delta>0$, the following inequality holds for any $x$, $\theta$, and $k\geq t_\delta$: 
\begin{align*}
	\|\mathbb{E}[F(X_k,\theta)\mid X_0=x]-\bar{F}(\theta)\|
	\leq 2L_1\delta(\|\theta\|+1).
\end{align*}
\end{lemma}

Lemma \ref{le:mixing_time} uses the mixing time to bound the bias (due to Markovian noise) in Algorithm (\ref{sa:algorithm}). See Section \ref{pf:le:mixing} for the proof. The next result uses the Lipschitz condition to control the difference between $\theta_{k_1}$ and $\theta_{k_2}$ when $|k_2-k_1|$ is not too large. 
\begin{lemma}\label{le_2}
For any $k_1<k_2$ satisfying $\alpha_{k_1,k_2-1}\leq \frac{1}{4L}$, the following two inequalities hold:
\begin{enumerate}[(1)]
	\item $\|\theta_{k_2}-\theta_{k_1}\|\leq 2L\alpha_{k_1,k_2-1}(\|\theta_{k_1}\|+1)$,
	\item  $\|\theta_{k_2}-\theta_{k_1}\|\leq 4L\alpha_{k_1,k_2-1}(\|\theta_{k_2}\|+1)$.
\end{enumerate}
\end{lemma}

The proof of Lemma \ref{le_2} is presented in
Section \ref{pf:le_2}. With the help of Lemmas \ref{le:mixing_time} and \ref{le_2}, we are now ready to bound the term $(c)$ in the following lemma. See Section \ref{pf:le_3} for the proof.
\begin{lemma}\label{le_3}
The following inequality holds for all $k$ such that $\alpha_{k-t_k,k-1}\leq \frac{1}{4L}$ (where we recall that $\alpha_{k-t_k,k-1}=\sum_{i=k-t_k}^{k-1}\alpha_i$): 
\begin{align*}
	(c)
	\leq 128L^2\alpha_k\alpha_{k-t_k,k-1}\left[\mathbb{E}_k[\|\theta_k-\theta^*\|^2]+(\|\theta^*\|+1)^2\right].
\end{align*}
\end{lemma}

Substituting the upper bounds we obtained for the terms $(a)-(d)$ into Eq. (\ref{eq:expand}), we have the following result, whose proof is presented in Section \ref{pf:le:recursion}.
\begin{lemma}\label{le:recursion}
It holds for all $k$ satisfying $\alpha_{k-t_k,k-1}\leq \frac{1}{4L}$ that:
\begin{align}
	\mathbb{E}[\|\theta_{k+1}-\theta^*\|^2]
	\leq\;&(1-2c_0\alpha_k+130L^2\alpha_k\alpha_{k-t_k,k-1})\mathbb{E}[\|\theta_k-\theta^*\|^2]\nonumber\\
	&+130L^2\alpha_k\alpha_{k-t_k,k-1}(\|\theta^*\|+1)^2.\label{eq:recursive_bound}
\end{align}
\end{lemma}

Eq. (\ref{eq:recursive_bound}) is of the desired recursive form (\ref{eq:hope-to-show}). Therefore, as long as the drift term dominates the error terms, i.e., $2c_0\alpha_k> 130L^2\alpha_k\alpha_{k-t_k,k-1}$, we can repeatedly use Eq. (\ref{eq:recursive_bound}) to derive finite-sample error bounds of Algorithm (\ref{sa:algorithm}). In particular, when Condition \ref{as:stepsize} is satisfied and $k\geq K$ (see Theorem \ref{thm:main} for the definition of $K$), we have by Eq. (\ref{eq:recursive_bound}) that
\begin{align*}
\mathbb{E}[\|\theta_{k+1}-\theta^*\|^2]
\leq(1-c_0\alpha_k)\mathbb{E}[\|\theta_k-\theta^*\|^2]+\beta_2\hat{\alpha}_k,
\end{align*}
where $\hat{\alpha}_k$ and $\beta_2$ are defined in Theorem \ref{thm:main}. Repeatedly using the preceding inequality starting from $K$, we obtain
\begin{align*}
\mathbb{E}[\|\theta_k-\theta^*\|^2]
\leq \mathbb{E}[\|\theta_K-\theta^*\|^2]\prod_{j=K}^{k-1}(1-c_0\alpha_j)+\beta_2\sum_{i=K}^{k-1}\hat{\alpha}_i\prod_{j=i+1}^{k-1}(1-c_0\alpha_{j}).
\end{align*}
To bound $\mathbb{E}[\|\theta_K-\theta^*\|^2]$, we use Lemma \ref{le_2} and $\alpha_{K-t_K,K-1}=\alpha_{0,K-1}\leq \frac{1}{4L}$ to obtain
\begin{align*}
\mathbb{E}[\|\theta_K-\theta^*\|^2]
\leq \mathbb{E}[(\|\theta_{K}-\theta_0\|+\|\theta^*-\theta_0\|)^2]\leq  \beta_1.
\end{align*}
The proof is now complete.

In this section, we have presented the finite-sample convergence bounds for SA algorithm (\ref{sa:algorithm}), and analyzed the convergence rates when using both constant and diminishing stepsizes. In the following section, to demonstrate the power of our SA results, we will use them to establish finite-sample guarantees for the $Q$-learning with linear function approximation algorithm that is frequently used to solve RL problems.

\section{Applications in Reinforcement Learning}\label{sec:RL}
We begin by introducing the Markov decision process (MDP) and the RL problem. An infinite horizon discounted MDP $\mathcal{M}$ is comprised by a tuple $(\mathcal{S},\mathcal{A},\mathcal{P},\mathcal{R},\gamma)$, where $\mathcal{S}$ is a set of states, $\mathcal{A}$ is a set of actions, $\mathcal{P}$ is the set of transition probabilities, $\mathcal{R}:\mathcal{S}\times\mathcal{A}\mapsto[0,r_{\max}]$ is the reward function, and $\gamma\in (0,1)$ is the discount factor. In this paper, we will work with MDPs with finite state-action spaces, i.e., $n:=|\mathcal{S}|<\infty$ and $m:=|\mathcal{A}|<\infty$. In this case, $\mathcal{P}=\{P_a\in\mathbb{R}^{n\times n}\mid a\in\mathcal{A}\}$ is a set of action-dependent transition probability matrices. The underlying model of the RL problem is essentially an MDP except that the transition probabilities and the reward function are unknown to the agent. 

The goal of RL is to find a policy for choosing actions based on the state of the environment such that the expected long-term reward is maximized. Formally, define the state-action value function (aka. the $Q$-function) of a policy $\pi$ at $(s,a)$ by $Q_\pi(s,a)=\mathbb{E}_\pi[\sum_{k=0}^{\infty}\gamma^k\mathcal{R}(S_k,A_k)\mid S_0=s,A_0=a]$, where we use the notation $\mathbb{E}_\pi[\,\cdot\,]$ to mean that the actions are chosen according to policy $\pi$, i.e., $A_k\sim \pi(\cdot|S_k)$ for all $k\geq 1$. Our goal is to find an optimal policy $\pi^*$ in the sense that its corresponding $Q$-function, denote by $Q^*$, satisfies $Q^*(s,a)\geq Q_\pi(s,a)$ for any $(s,a)$ and $\pi$. A fundamental property of the function $Q^*$ is that, if one simply selects actions greedy based on $Q^*$, then that is an optimal policy. That is, $\pi^*(s)\in\arg\max_{a\in\mathcal{A}}Q^*(s,a)$ for all state $s\in\mathcal{S}$ \cite{bertsekas1996neuro}. Therefore, solving the RL problem reduces to finding the optimal $Q$-function.

\subsection{Q-Learning with Linear Function Approximation}\label{subsec:RL:overview}

The $Q$-learning algorithm proposed in \cite{watkins1992q} is a popular approach for estimating the function $Q^*$. However, it is well-known that a major challenge in $Q$-learning is that the algorithm becomes intractable when the number of state-action pairs is large. Therefore, we consider approximating the optimal $Q$-function from a chosen function space. We next describe the approximation architecture.

Let $\phi_i\in\mathbb{R}^{mn}$, $1\leq i\leq d$ be a set of basis vectors. Denote $\phi(s,a)=[\phi_1(s,a),\cdots,\phi_d(s,a)]^\top$, which is a column vector. We assume without loss of generality that the basis vectors $\{\phi_i\}_{1\leq i\leq d}$ are linearly independent and are normalized so that $\|\phi(s,a)\|\leq 1$ for all $(s,a)$. Define the matrix $\Phi\in\mathbb{R}^{mn\times d}$ by
\begin{align*}
	\Phi = \left[\begin{array}{ccc}
		\vert     &  & \vert\\
		\phi_1     & ... & \phi_d\\
		\vert & & \vert
	\end{array}\right] = \left[\begin{array}{ccc}
		\mbox{---}   & \phi(s_{1},a_{1})^\top & \mbox{---}\\
		...    & ... & ...\\
		\mbox{---}   & \phi(s_{n},a_{m})^\top & \mbox{---}
	\end{array}\right].
\end{align*}
Then the linear subspace $\mathcal{W}$ spanned by the basis vectors $\{\phi_i\}$ can be written as $\mathcal{W}=\{\tilde{Q}_{\theta} =  \Phi \theta\mid \theta\in\mathbb{R}^d\}$. We will use $\mathcal{W}$ as our approximating function space, and the goal here is to find $\theta^*$ such that $\tilde{Q}_{\theta^*}$ best approximates $Q^*$. 

Using the notation above, we now present the $Q$-learning algorithm under linear function approximation \cite{bertsekas1996neuro}. Let $\{(S_k,A_k)\}$ be a sample trajectory generated by applying some \textit{behavior} policy $\pi$ to the system model. Note that $\{(S_k,A_k)\}$ forms a Markov chain. Then, the parameter $\theta$ of the approximation $\tilde{Q}_\theta$ is updated by:
\begin{align}\label{algorithm:Q-learning}
\theta_{k+1}=\theta_k+\alpha_k\phi(S_k,A_k)\Delta(\theta_k,S_k,A_k,S_{k+1}),
\end{align}
where $\Delta:\mathbb{R}^d\times\mathcal{S}\times\mathcal{A}\times\mathcal{S}\mapsto\mathbb{R}$ is defined by
\begin{align*}
\Delta(\theta,s,a,s')=\mathcal{R}(s,a)+\gamma\max_{a' \in \mathcal{A}}\phi(s',a')^\top\theta-\phi(s,a)^\top\theta
\end{align*}
for all $\theta$ and $(s,a,s')$, and presents the temporal difference. Algorithm (\ref{algorithm:Q-learning}) can be viewed as an SA algorithm for solving the equation
\begin{align}\label{eq:pbj}
\mathbb{E}_{S\sim \mu_S(\cdot),A\sim\pi(\cdot|S),S'\sim P_{A}(S,\cdot)}[\phi(S,A)\Delta(\theta,S,A,S')]=0,
\end{align}
where $\mu_S$ stands for the stationary distribution of the Markov chain $\{S_k\}$ under policy $\pi$ (provided that it exists and is unique). Under some mild conditions, Eq. (\ref{eq:pbj}) is equivalent to a so-called \textit{projected Bellman equation} \cite{melo2008analysis}. In the special case where the feature matrix $\Phi$ is an identity matrix with dimension $mn$, Algorithm (\ref{algorithm:Q-learning}) reduces to the tabular $Q$-learning algorithm \cite{watkins1992q}, and Eq. (\ref{eq:pbj}) becomes the regular Bellman equation for $Q^*$ \cite{bertsekas1996neuro}.

In general, Eq. (\ref{eq:pbj}) may not necessarily admit a solution, see Appendix \ref{ap:no_solution} for such an example, and the iteration in (\ref{algorithm:Q-learning}) may diverge \cite{baird1995residual}. 
However, it was shown in \cite{melo2008analysis} that under an assumption on the behavior policy $\pi$, $\theta_k$ converges to the solution of Eq. (\ref{eq:pbj}), denoted by $\theta^*$, almost surely. In this paper, we work with a similar condition, and focus on establishing the finite-sample bounds of Algorithm (\ref{algorithm:Q-learning}). We begin by stating our assumptions.
\begin{assumption}\label{as:q-markov-chain}
The behavior policy $\pi$ satisfies $\pi(a|s)>0$ for all $(s,a)$, and the Markov chain $\{S_k\}$ induced by $\pi$ is irreducible and aperiodic. 
\end{assumption}

Assumption \ref{as:q-markov-chain} essentially requires that the behavior policy $\pi$ has enough exploration, and is standard in studying off-policy value-based RL algorithms  \cite{tsitsiklis1997analysis,tsitsiklis1999average}. Under Assumption \ref{as:q-markov-chain}, the Markov chain $\{S_k\}$ has a unique stationary distribution, which we have denoted by $\mu_S$. In addition, since the state-space $\mathcal{S}$ is finite, the Markov chain $\{S_k\}$ mixes geometrically fast in that there exist $C'\geq 1$ and $\rho'\in (0,1)$ such that $\max_{s\in\mathcal{S}}\|P^k(s,\cdot)-\mu_S(\cdot)\|_{\text{TV}}\leq C'\rho'^k$ for all $k\geq 0$ \cite{levin2017markov}.
\begin{assumption}\label{as:q-behhavior-policy}
The target equation (\ref{eq:pbj}) has a unique solution $\theta^*$, and there exists $\kappa>0$ such that the following inequality holds for all $\theta\in\mathbb{R}^d$:
\begin{align}\label{key}
	\gamma^2\mathbb{E}_{\mu_S}[\max_{a \in\mathcal{A}}\tilde{Q}_\theta(S,a)^2]-\mathbb{E}_{\mu_S,\pi}[\tilde{Q}_\theta(S,A)^2]
	\leq -\kappa \|\theta\|^2.
\end{align}
\end{assumption}

We make Assumption \ref{as:q-behhavior-policy} and especially Eq. (\ref{key}) to ensure the stability of Algorithm (\ref{algorithm:Q-learning}), which is in the same spirit to the conditions proposed in \cite{melo2008analysis}. A detailed discussion about this assumption and comparison to related conditions are presented in Section \ref{subsec:RL:discussion}.

\subsection{Finite-Sample Convergence Guarantees}\label{subsec:RL:finite-sample bounds}
To apply our SA results, we begin by modeling Algorithm (\ref{algorithm:Q-learning}) in the form of SA algorithm (\ref{sa:algorithm}). Define  $X_k=(S_k,A_k,S_{k+1})$ for all $k\geq 0$.  It is clear that $\{X_k\}$ is also a Markov chain with finite state-space $\mathcal{X}=\{(s,a,s')\mid s \in \mathcal{S},\pi(a|s)>0,P_a(s,s')>0\}$. Moreover, under Assumption \ref{as:q-markov-chain}, we will show that the Markov chain $\{X_k\}$ also has a unique stationary distribution, which we denote by $\mu_X$ and is given by $\mu_X(s,a,s')=\mu_S(s)\pi(a|s)P_a(s,s')$ for all $(s,a,s')\in\mathcal{X}$.
Define an operator $F:\mathcal{S}\times\mathcal{A}\times\mathcal{S}\times\mathbb{R}^d\mapsto\mathbb{R}^d$ by
\begin{align}\label{def:F-q-learning}
F(x,\theta)=F(s,a,s',\theta)=\phi(s,a)\Delta(\theta,s,a,s')
\end{align}
for all $\theta$ and $x=(s,a,s')$. Then Algorithm \eqref{algorithm:Q-learning} can be written in the same form as the SA algorithm (\ref{sa:algorithm}) with the additive noise $w_k$ being identically equal to zero. Let $\bar{F}(\theta)=\mathbb{E}_{\mu_X}[F(X,\theta)]$. We see that $\bar{F}(\theta)=0$ is exactly the targeting equation (\ref{eq:pbj}).

To apply Theorem \ref{thm:main}, we first show in the following proposition that Assumptions \ref{as:Lipschitz}, \ref{as:stability}, and \ref{as:markov-chain} are satisfied in the context of $Q$-learning. The proof is presented in Section \ref{pf:thm:apply-to-q-learning}.
\begin{proposition}\label{thm:apply-to-q-learning}
Suppose that Assumptions \ref{as:q-markov-chain} and \ref{as:q-behhavior-policy} are satisfied, then we have the following results.
\begin{enumerate}[(1)]
	\item The Markov chain $\{X_k\}$ is irreducible and aperiodic, hence having a unique stationary distribution $\mu_X$. In addition, we have $\max_{x\in\mathcal{X}}\|P^{k+1}(x,\cdot)-\mu_X(\cdot)\|_{\text{TV}}\leq C'\rho'^k$ for all $k\geq 0$.
	\item Let $M= 1+\gamma+r_{\max}$. Then we have (a) $\|F(x,\theta_1)-F(x,\theta_2)\|\leq M\|\theta_1-\theta_2\|$ for all $x$, $\theta_1$, and $\theta_2$, and (b) $\|F(x,0)\|\leq M$ for all $x$.
	\item The equation $\bar{F}(\theta)=0$ has a unique solution $\theta^*$, and we have $(\theta-\theta^*)^\top \bar{F}(\theta)\leq -\frac{\kappa}{2}\|\theta-\theta^*\|^2$ for all $\theta\in\mathbb{R}^d$.
\end{enumerate}
\end{proposition}

Similarly as in Section \ref{sec:sa}, given precision $\delta>0$, we define $t_\delta$ as the mixing time of the Markov chain $\{X_k\}$ with precision $\delta>0$. Observe that Proposition \ref{thm:apply-to-q-learning} (1) implies that there exists a constant $M_1=\frac{\log(C'/\rho')}{\log(1/\rho')}$ such that $t_\delta\leq M_1(\log(1/\delta)+1)$ for any $\delta>0$. This is analogous to Eq. (\ref{eq:mixing_time_bound}) in Section \ref{sec:sa}.

We next use Theorem \ref{thm:main} to establish the finite-sample bounds of the $Q$-learning algorithm (\ref{algorithm:Q-learning}). In the diminishing stepsize regime, we only present case (1) (c) of Corollary \ref{thm:diminishing_step_size}, which has the best convergence rate. Let $\eta_1=(\|\theta_0\|+\|\theta_0-\theta^*\|+1)^2$ and $\eta_2=130M^2(\|\theta^*\|+1)^2$. The following theorem is a direct implication of Theorem \ref{thm:main}, hence we omit its proof.

\begin{theorem}\label{thm:q-constant}
Consider $\{\theta_k\}$ of the $Q$-learning algorithm (\ref{algorithm:Q-learning}). Suppose that Assumptions \ref{as:q-markov-chain} and \ref{as:q-behhavior-policy} are satisfied, Then we have the following results.
\begin{enumerate}[(1)]
	\item When $\alpha_k\equiv\alpha$ with $\alpha$ chosen such that $\alpha t_\alpha\leq \frac{\kappa}{260M^2}$, we have for all $k\geq t_\alpha$:
	\begin{align*}
		\mathbb{E}[\|\theta_{k}-\theta^*\|^2]\leq \eta_1\left(1-\kappa\alpha/2\right)^{k-t_\alpha}+2\eta_2\alpha t_\alpha/\kappa.
	\end{align*}
\item When $\alpha_k=\alpha/(k+h)$, where $\alpha>2/\kappa$ and $h$ is large enough, there exists $K'>0$ such that we have for all $k\geq K'$:
\begin{align*}
	\mathbb{E}[\|\theta_k-\theta^*\|^2]\leq\eta_1\left(\frac{K'+h}{k+h}\right)^{\frac{\kappa\alpha}{2}}+\frac{16e\eta_2\alpha^2M_1}{\kappa\alpha-2} \frac{\left[\log\left(\frac{k+h}{\alpha}\right)+1\right]}{k+h}.
\end{align*}
\end{enumerate}
\end{theorem}

Theorem \ref{thm:q-constant} (1) is qualitatively similar to Corollary \ref{thm:constant_step_size} in that the iterates of $Q$-learning converge exponentially fast to a ball centered at $\theta^*$, and the size of the ball is proportional to $\alpha t_\alpha$. 
This agrees with results in \cite{srikant2019finite,bhandari2018finite}, where the popular TD-learning with linear function approximation algorithm was studied. Theorem \ref{thm:q-constant} (2) suggests that for properly chosen diminishing stepsizes, the optimal convergence rate is roughly $O(\log (k)/k)$. The $\log(k)$ factor is a consequence of performing Markovian sampling of $\{(S_k,A_k)\}$.

\subsection{Discussion about Assumption \ref{as:q-behhavior-policy} on the Behavior Policy}\label{subsec:RL:discussion}
In this section, we take a closer look at Assumption \ref{as:q-behhavior-policy} and especially Eq. (\ref{key}), which is made for the stability of the $Q$-learning with linear function approximation algorithm. First note that Eq. \eqref{key} is equivalent to
\begin{align}\label{con}
\gamma^2\mathbb{E}_{\mu_S}[\max_{a \in\mathcal{A}}\tilde{Q}_\theta(S,a)^2]<\mathbb{E}_{\mu_S,\pi}[\tilde{Q}_\theta(S,A)^2]
\end{align}
for all nonzero $\theta$.
The direction Eq. (\ref{key}) implying Eq. (\ref{con}) is trivial. As for the other direction, let
\begin{align*}
\kappa = -\max_{\theta:\|\theta\|=1}\{\gamma^2\mathbb{E}_{\mu_S}[\max_{a\in\mathcal{A}}\tilde{Q}_\theta(S,a)^2]-\mathbb{E}_{\mu_S,\pi}[\tilde{Q}_\theta(S,A)^2]\}.
\end{align*}
By Weierstrass extreme value theorem \cite{rudin1964principles}, $\kappa$ is well-defined and strictly positive because it is the maximum of a continuous function over a compact set. This immediately gives Eq. \eqref{key}.

Similar assumptions on the behavior policy were also proposed in \cite{melo2008analysis,lee2019unified}. Although the exact form of the conditions are different,  they all follow the same spirit. That is, with a chosen Lyapunov function, the condition should enable us to show that the corresponding ODE
\begin{align}\label{ODE-q-learning}
\dot{\theta}(t)=\bar{F}(\theta(t))
\end{align}
of the $Q$-learning algorithm (\ref{algorithm:Q-learning}) is globally asymptotically stable (GAS). We next briefly compare our condition to those proposed in \cite{melo2008analysis,lee2019unified}. The condition (Eq. (7)) in \cite{melo2008analysis}  implies
\begin{align}\label{eq:melo}
2\gamma^2\mathbb{E}_{\mu_S}[(\max_{a\in\mathcal{A}}\tilde{Q}_\theta(S,a))^2]<\mathbb{E}_{\mu_S,\pi}[\tilde{Q}_\theta(S,A)^2]
\end{align}
for all nonzero $\theta$ \footnote{The factor of 2 appears to be missing in \cite{melo2008analysis}.}. The RHS is the same for both Eqs. (\ref{eq:melo}) and (\ref{con}). On the LHS, Eq. (\ref{eq:melo}) has an additional factor of $2$, and the square is outside the max operator. Although they are similar, our condition and the condition proposed in \cite{melo2008analysis} do not imply each other.
As for the condition proposed in \cite{lee2019unified}, while it is not clear if it is less restrictive than ours, it is shown that the condition in \cite{lee2019unified} implies the condition in \cite{melo2008analysis} under more restrictive assumptions. However, \cite{lee2019unified} assumes i.i.d. sampling, and studies only the asymptotic convergence rather than finite-sample error bounds.

We next analyze how the discount factor, the basis vectors $\{\phi_i\}$, and the behavior policy $\pi$ impact condition (\ref{con}). In terms of the dependence on the discount factor, it is clear that condition (\ref{con}) is easier to satisfy for smaller discount factor. This agrees with our numerical simulations provided in Section \ref{subsec:RL:experiments}. Utility of smaller discount factors in RL was also noted
in \cite{jiang2015dependence}, albeit in a completely different context of generalization. To see the impact of the basis vectors and the behavior policy, consider the following two examples.

\textbf{Uni-Dimension Case.}
Suppose that $d=1$. That is, there is only one basis vector $\phi_1$, and the weight $\theta$ is a scalar. Condition (\ref{con}) reduces to
\begin{align}\label{1d-weaker}
	\gamma^2\mathbb{E}_{\mu_S}[\max_{a \in \mathcal{A}}\phi(S,a)^2]<\mathbb{E}_{\mu_S,\pi}[\phi(S,A)^2].
\end{align}
Define $h^+=\mathbb{E}_{\mu_S,\pi}[\gamma\phi(S,A)\max_{a' \in \mathcal{A}}\phi(S',a')-\phi(S,A)^2]$, $h^-=\mathbb{E}_{\mu_S,\pi}[\gamma\phi(S,A)\min_{a' \in \mathcal{A}}\phi(S',a')-\phi(S,A)^2]$, and $r_\pi=\mathbb{E}_{\mu_S,\pi}[\phi(S,A)\mathcal{R}(S,A)]$. Then we have the following result. See Section \ref{subsec:pf:thm:1d-iff} for the proof.

\begin{proposition}\label{thm:1d-iff}
	Eq. (\ref{1d-weaker}) implies $h^+<0$ and $h^-<0$, and the following statements regarding the relation between the stability of ODE (\ref{ODE-q-learning}) and the sign of $h^+$ and $h^-$ hold:
	\begin{align*}
		\text{ODE (\ref{ODE-q-learning}) is GAS}\Longleftrightarrow \begin{cases}
			h^+<0,h^-<0,&\text{when }r_{\pi}=0,\\
			h^+<0,h^-\leq 0,&\text{when }r_{\pi}>0,\\
			h^+\leq 0,h^-<0,&\text{when }r_{\pi}<0.
		\end{cases} 
	\end{align*}
\end{proposition}

Proposition \ref{thm:1d-iff} implies that the condition (\ref{1d-weaker}) is ``almost necessary''  for the GAS of ODE (\ref{ODE-q-learning}). Moreover, it is clear from Eq. (\ref{1d-weaker}) that when $d=1$, there always exists a behavior policy $\pi$ such that Eq. (\ref{1d-weaker}) is satisfied. For example, $\pi(s)\in\arg\max_{a \in\mathcal{A}}\phi(s,a)^2$ is a feasible behavior policy.

\textbf{Full-Dimension Case.} Suppose that $d=mn$, i.e., there is no dimension reduction at all. We want to emphasize that this is not equivalent to the tabular $Q$-learning. Even when $\Phi$ is a full-rank square matrix, the $Q$-learning with linear function approximation algorithm does not coincide with the tabular $Q$-learning algorithm. In fact, the divergence counter-example provided in \cite{baird1995residual} belongs to this setting. We show in the following proposition that, in the full-dimension case, condition (\ref{con}) is feasible in terms of the behavior policy $\pi$ only when the discount factor $\gamma$ is sufficiently small. See Section \ref{pf:thm:full-dimension} for its proof.
\begin{proposition}\label{thm:full-dimension}
	When $d=mn$ and $\gamma^2\geq 1/m$, condition (\ref{con}) is infeasible for any behavior policy $\pi$.
\end{proposition}

We now compare the results for the two extreme cases, i.e., $d=1$ and $d=mn$. We see that in the uni-dimensional case, Eq. (\ref{con}) implies a condition which is almost sufficient and necessary for the GAS of the equilibrium $\theta^*$ to ODE (\ref{ODE-q-learning}). Moreover, there always exists a behavior policy $\pi$ satisfying $(\ref{con})$. However, in the full-dimensional case, condition (\ref{con}) is infeasible in terms of the behavior policy $\pi$ when $\gamma^2\geq 1/m$, which can usually happen in practice.

\subsection{Numerical Experiments}\label{subsec:RL:experiments}
In this section, we present numerical experiments to demonstrate the sufficiency of condition (\ref{con}), as well as the resulting convergence rates of the $Q$-learning with linear function approximation algorithm.

We begin by verifying the sufficiency of condition (\ref{key}). Let
\begin{align}\label{key2}
\omega(\pi)=\min_{\{\theta:\|\theta\|=1\}}\frac{\mathbb{E}_{\mu_S,\pi}[\tilde{Q}_\theta(S,A)^2]}{\mathbb{E}_{\mu_S}[\max_{a \in \mathcal{A}}\tilde{Q}_\theta(S,a)^2]}.
\end{align}
Then condition (\ref{key}) is equivalent to $\omega(\pi)>\gamma^2$. One way to compute $\omega(\pi)$ is presented in Section \ref{computing_omega}.

In our simulation, we consider the divergent example of $Q$-learning with linear function approximation introduced in \cite{baird1995residual}, which is an MDP with  $7$ states and $2$ actions. 
To demonstrate the effectiveness of condition (\ref{key}) for the stability of $Q$-learning, in our first set of simulations, the reward function is set to zero.  Since the reward function is identically zero, $Q^*$ is zero, implying $\theta^*$ is zero. We choose the behavior policy $\pi$ which takes each action with equal probability. In this case, we have $\omega(\pi)\approx 0.5$, giving the threshold for $\gamma$ to satisfy Eq. (\ref{key}) being $\omega(\pi)^{1/2}\approx 0.7$. In our simulation, we choose constant stepsize $\alpha=0.01$,  discount factor $\gamma \in \{0.7,0.9,0.97\}$, and plot $\|\theta_k\|$  as a function of the number of iterations $k$ in Fig. \ref{fig:1}. Here, $\theta_{k}$ converges when $\gamma=0.7,0.9$,  but diverges when $\gamma=0.97$. This demonstrates that condition \eqref{key} is sufficient but not necessary for convergence. This also shows that when Eq. \eqref{key} is satisfied, the counter-example from \cite{baird1995residual} converges.

\begin{figure}[h]
\centering
\begin{minipage}{.40\textwidth}
	\centering
	\includegraphics[width=\linewidth]{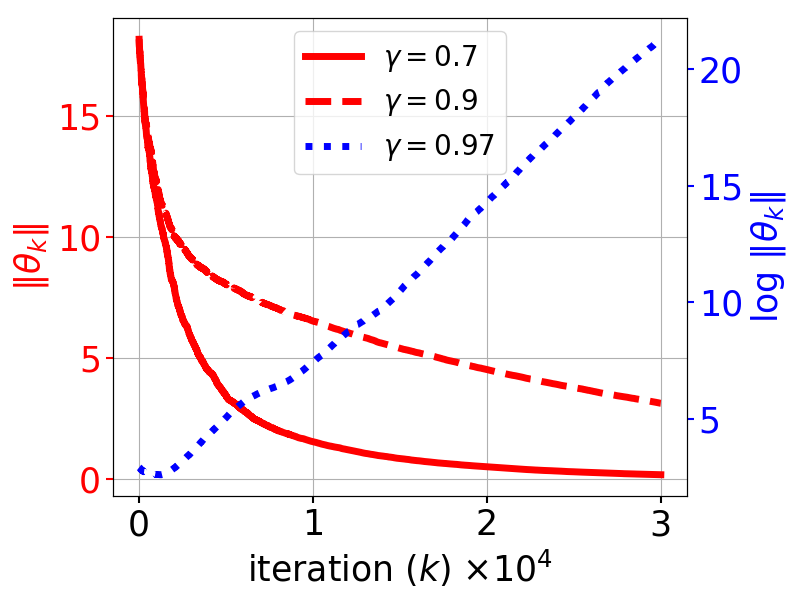}
	\caption{{\small Convergence of $Q$-learning with linear function approximation for different discount factor $\gamma$}}
	\label{fig:1}
\end{minipage}
\hfill
\begin{minipage}{.40\textwidth}
	\centering
	\includegraphics[width=\linewidth]{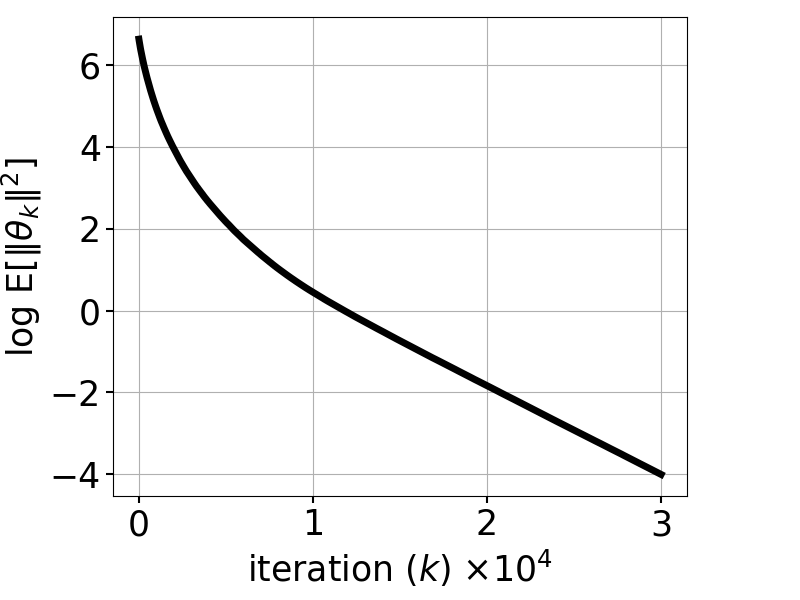}
	\caption{{\small Exponentially fast convergence of $Q$-learning with linear function approximation for $\gamma=0.7$}}
	\label{fig:2}
\end{minipage}
\end{figure}

To show the exponential convergence rate for using constant stepsize,  we consider the convergence of $\theta_{k}$ when $\gamma=0.7$ given in Fig. \ref{fig:2}, where we plot $\log{\mathbb{E}[\|\theta_k\|^2]}$ as a function of the number of iterations $k$. In this case, $\theta_{k}$ seems to converge geometrically, which agrees with Theorem \ref{thm:q-constant} (1).

We next numerically verify the convergence rates of $Q$-learning with linear function approximation for using diminishing stepsizes $\alpha_k=\alpha/(k+h)^\xi$. We use the same MDP model and behavior policy. The only difference is that the reward is no longer set to zero, but is sampled independently from a uniform distribution on $(0,1)$ for all state-action pairs. The constant $\kappa$ given in Eq. (\ref{key}) is estimated by numerical optimization, and the discount factor $\gamma$ is set to be $0.7$ to ensure convergence. In Fig. \ref{fig:3}, we plot $\mathbb{E}[\|\theta_k-\theta^*\|^2]$ as a function of $k$ for $\xi\in \{0.4,0.6,0.8,1\}$. In the case where $\xi=1$, the constant coefficient $\alpha$ is chosen such that $\kappa\alpha\geq 2$ in order to achieve the optimal  convergence rate. We see that the iterates converge for all $\xi\in(0,1]$. Moreover, the larger the value of $\xi$ is, the faster $\theta_k$ converges. 

\begin{figure}[h]
\centering
\begin{minipage}{.40\textwidth}
	\centering
	\includegraphics[width=\linewidth]{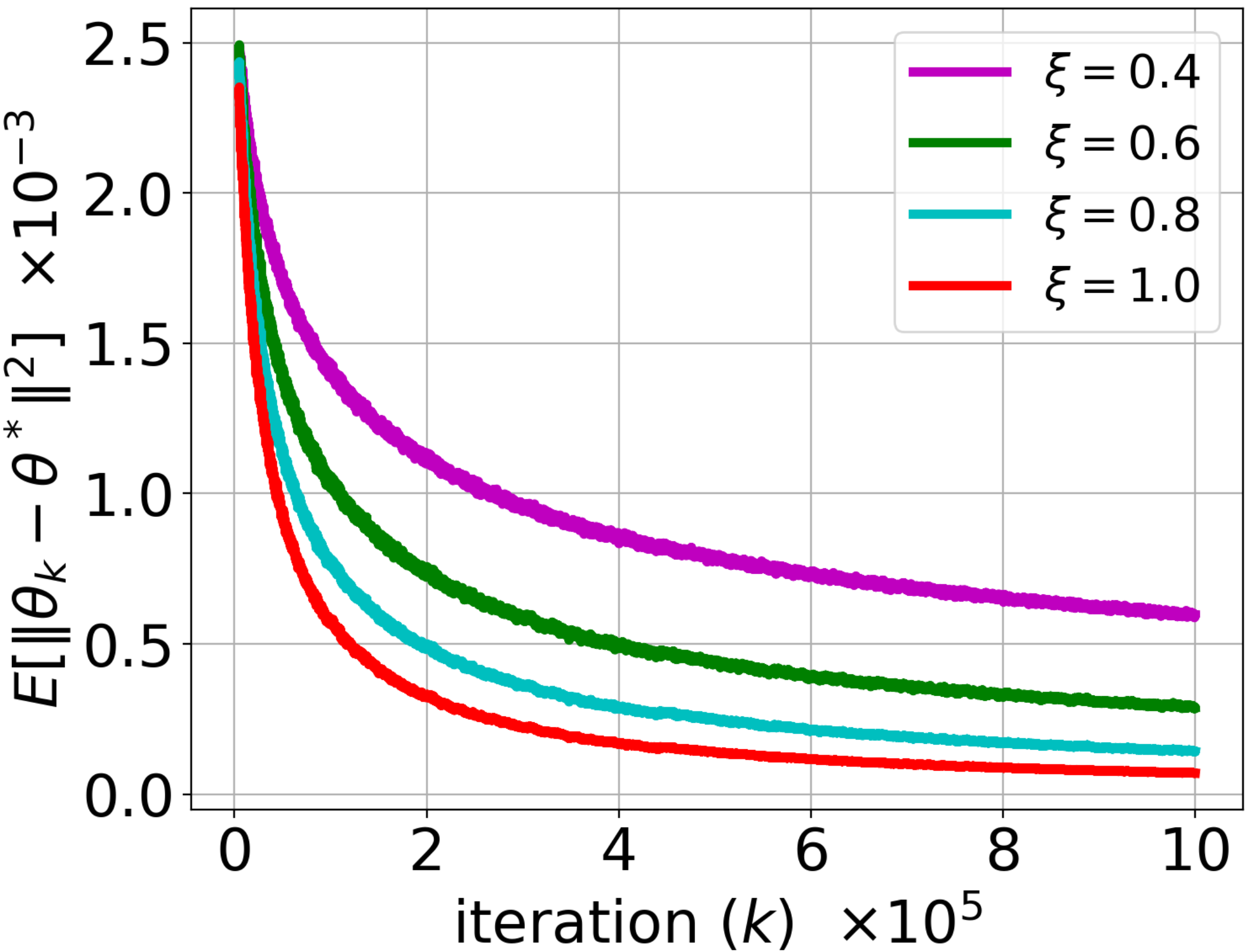}
	\caption{{\small Convergence for diminishing stepsizes}}
	\label{fig:3}
\end{minipage}
\hfill
\begin{minipage}{.40\textwidth}
	\centering
	\includegraphics[width=\linewidth]{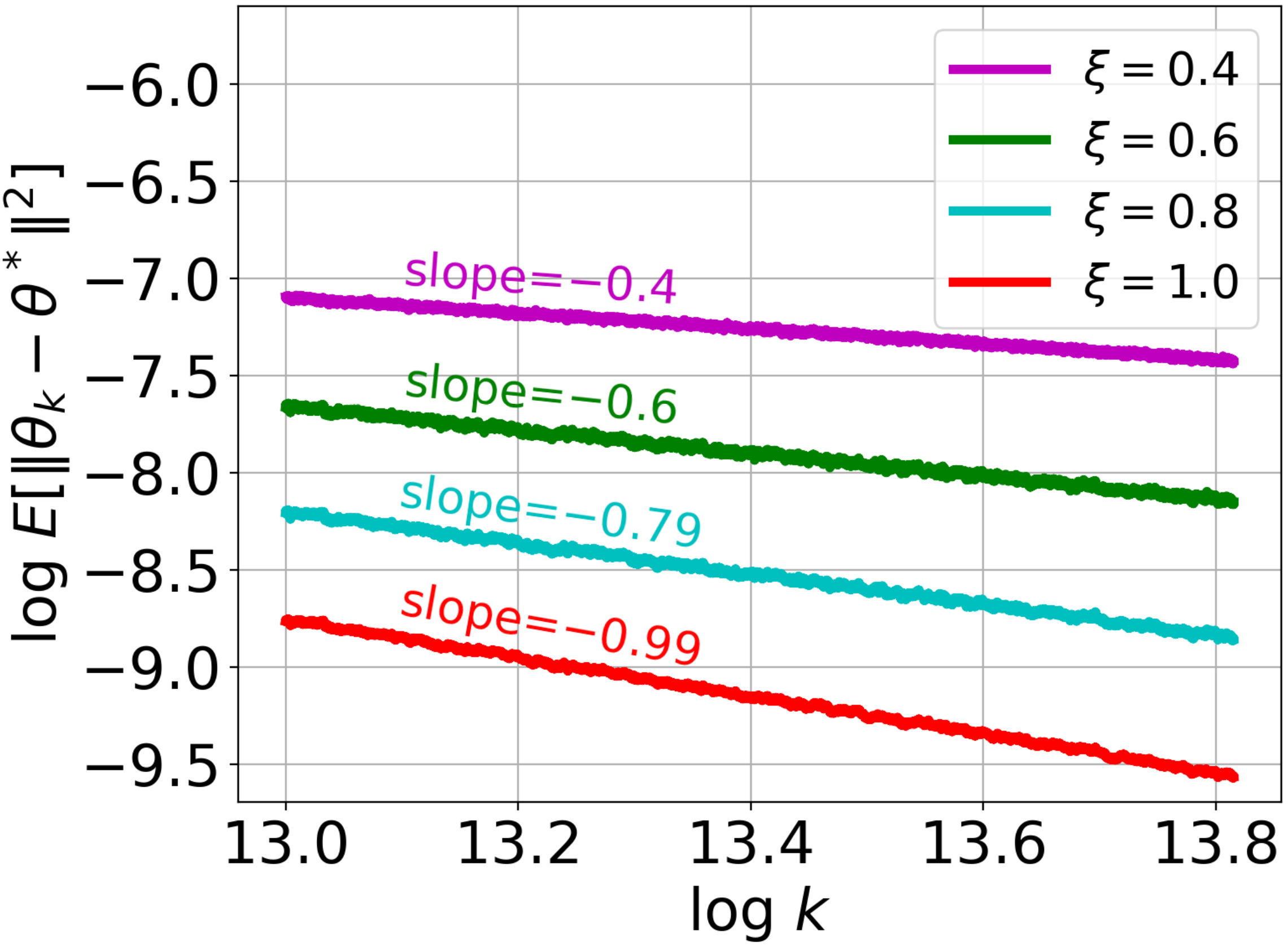}
	\caption{{\small Asymptotic convergence rate for diminishing stepsizes}}
	\label{fig:4}
\end{minipage}
\end{figure}

To further verify the convergence rates, we plot $\log\mathbb{E}[\|\theta_k-\theta^*\|^2]$ as a function of $\log k$ in Fig. \ref{fig:4} and look at its asymptotic behavior. We see that the slope is approximately $-\xi$, which agrees with Theorem \ref{thm:q-constant} (2). 

In addition to the MDP used in Baird's counter-example \cite{baird1995residual}, numerical simulations corresponding to a larger MDP are presented in Appendix \ref{ap:larger}, and the results are consistent with the theory as well as the outcomes of the simulations in this section.

\section{Conclusion}
In this paper we establish finite-sample convergence guarantees for a general nonlinear SA algorithm with Markovian noise. We then use it to derive, for the first time, finite-sample bounds for the $Q$-learning with linear function approximation algorithm. Since such an algorithm is known to diverge in general, we study it under a condition on the basis vectors, the behavior policy, and the discount factor that ensure stability. Sufficiency of this condition and the rate of convergence of $Q$-learning are verified numerically in the context of a well-known example.

\bibliographystyle{imsart-number}
\bibliography{references}    

\newpage

\begin{appendix}
\section{Proof of All Technical Results in Section \ref{sec:sa}}

\subsection{Proof of Corollary \ref{thm:constant_step_size}}\label{pf:thm:constant_step_size}
When $\alpha_k=\alpha$ for all $k\geq0$, since $K=t_\alpha$, we have (1) $\prod_{j=K}^{k-1}(1-c_0\alpha_j)= (1-c_0 \alpha)^{k-t_\alpha}$, and (2) $\sum_{i=K}^{k-1}\hat{\alpha}_i\prod_{j=i+1}^{k-1}(1-c_0\alpha_j)=\alpha^2t_\alpha \sum_{i=t_\alpha}^{k-1}(1-c_0\alpha)^{k-i-1}\leq \frac{\alpha t_\alpha}{c_0}$. This proves the result.

{\subsection{Proof of Corollary \ref{thm:diminishing_step_size}}\label{pf:thm:diminishing_step_size}
We first verify Condition \ref{as:stepsize}. When $\alpha_k=\alpha/(k+h)^\xi$, using Eq. (\ref{eq:mixing_time_bound}) and we have
\begin{align*}
	t_k\leq L_3(\log(1/\alpha_k)+1)=L_3(\xi\log(k+h)+\log(1/\alpha)).
\end{align*}
It follows that
\begin{align*}
	\alpha_{k-t_k,k-1}&\leq t_k\alpha_{k-t_k}\\
	&\leq L_3(\log(1/\alpha_k)+1)\frac{\alpha}{(k-t_k+h)^\xi}\\
	&\leq \frac{\alpha L_3(\log(1/\alpha_k)+1)}{(k-L_3(\log(1/\alpha_k)+1)+h)^\xi}\\
	&= \frac{\alpha L_3(\log(1/\alpha_k)+1)}{(k-L_3(\xi\log(k+h)+\log(1/\alpha))+1)+h)^\xi}\\
	&= \frac{\alpha_kL_3(\log(1/\alpha_k)+1)(k+h)^\xi }{(k-L_3(\xi\log(k+h)+\log(1/\alpha))+1)+h)^\xi},
\end{align*}
where the last line follows from multiplying $\alpha_k$ and dividing $\frac{\alpha}{(k+h)^\xi}$. Therefore, we have
\begin{align*}
	\lim_{(k+h)\rightarrow\infty}\frac{\alpha_{k-t_k,k-1}}{\alpha_kL_3(\log(1/\alpha_k)+1)}=1,
\end{align*}
which implies that there exists $\bar{h}_1=\bar{h}_1(\alpha,\xi)>0$ such that $\alpha_{k-t_k,k-1}\leq 2L_3 \alpha_k(\log(1/\alpha_k)+1)$ for all $k\geq 0$. Since we also have $\lim_{(k+h)\rightarrow\infty}\alpha_k(\log(1/\alpha_k)+1)=0$,
there exists $\bar{h}_2=\bar{h}_2(\alpha,\xi)>0$ such that $\alpha_{k-t_k,k-1}\leq \frac{c_0}{130L^2}$ for all $k\geq t_k$. Let $\bar{h}=\max(\bar{h}_1,\bar{h}_2)$. Then when $h\geq \bar{h}$, Condition \ref{as:stepsize} is satisfied for any $k\geq 0$. Moreover, we have in this case $\alpha_{k-t_k,k-1}\leq 2L_3 \alpha_k(\log(1/\alpha_k)+1)$ for any $k\geq 0$. This is useful for us to derive the explicit convergence rate in the following.

To prove Corollary \ref{thm:diminishing_step_size}, we begin by simplifying Eq. (\ref{eq:sa-bound}) of Theorem \ref{thm:main}. For $k\geq K$, we have
\begin{align}
	\mathbb{E}[\|\theta_{k}-\theta^*\|^2]
	\leq \;&\beta_1\prod_{j=K}^{k-1}(1-c_0 \alpha_j)+\beta_2\sum_{i=K}^{k-1}\alpha_i\alpha_{i-k_i,i-1}\prod_{j=i+1}^{k-1}(1-c_0\alpha_j)\nonumber\\
	\leq \;&\beta_1\prod_{j=K}^{k-1}(1-c_0 \alpha_j)+\beta_2 \sum_{i=K}^{k-1}2L_3\alpha_i^2(\log(1/\alpha_i)+1)\prod_{j=i+1}^{k-1}(1-c_0\alpha_j)\nonumber\\
	\leq \;&\beta_1\prod_{j=K}^{k-1}(1-c_0\alpha_j)+2L_3\beta_2(\log(1/\alpha_k)+1) \sum_{i=K}^{k-1}\alpha_i^2\prod_{j=i+1}^{k-1}(1-c_0\alpha_j)\nonumber\\
	\leq \;&\beta_1\underbrace{\prod_{j=K}^{k-1}(1-c_0\alpha_j)}_{A_1}+2L_3\beta_2\left(\log\left(\frac{k+h}{\alpha}\right)+1\right)\underbrace{\sum_{i=K}^{k-1}\alpha_i^2\prod_{j=i+1}^{k-1}(1-c_0\alpha_j)}_{A_2}.\label{eq:sa-rate}
\end{align}
We next bound the terms $A_1$ and $A_2$. For the term $A_1$, we have
\begin{align}\label{eq:A_1}
	A_1&=\prod_{j=K}^{k-1}\left(1-\frac{c_0\alpha}{(j+h)^\xi}\right)\nonumber\\
	&\leq \exp\left(-c_0\alpha\sum_{j=K}^{k-1}\frac{1}{(j+h)^\xi}\right)\nonumber\\
	&\leq \exp\left(-c_0\alpha\int_{K}^{k}\frac{1}{(x+h)^\xi}dx\right)\nonumber\\
	&\leq \begin{dcases}
		\left(\frac{K+h}{k+h}\right)^{c_0\alpha},&\xi=1,\\
		e^{-\frac{c_0\alpha}{1-\xi}\left((k+h)^{1-\xi}-(K+h)^{1-\xi}\right)},&\xi\in (0,1).
	\end{dcases}
\end{align}
We next consider the term $A_2$. When $\xi=1$, using the same line of analysis for evaluating $A_1$, we have
\begin{align*}
	A_2&=\sum_{i=K}^{k-1}\frac{\alpha^2}{(i+h)^2}\prod_{j=i+1}^{k-1}\left(1-\frac{c_0\alpha}{j+h}\right)\\
	&\leq  \sum_{i=K}^{k-1}\frac{\alpha^2}{(i+h)^{2}}\left(\frac{i+1+h}{k+h}\right)^{c_0\alpha}\tag{Same to Eq. (\ref{eq:A_1})}\\
	&=\frac{\alpha^2}{(k+h)^{c_0\alpha}}\sum_{i=K}^{k-1}\left(\frac{i+1+h}{i+h}\right)^2\frac{1}{(i+1+h)^{2-c_0\alpha}}\\
	&=\frac{\alpha^2}{(k+h)^{c_0\alpha}}\sum_{i=K}^{k-1}\left(1+\frac{1}{i+h}\right)^2\frac{1}{(i+1+h)^{2-c_0\alpha}}\\
	&\leq \frac{4\alpha^2}{(k+h)^{c_0\alpha}}\underbrace{\sum_{i=K}^{k-1}\frac{1}{(i+1+h)^{2-c_0\alpha}}}_{A_3}.
\end{align*}
The upper bound for the term $A_3$ depends on the value $c_0\alpha$. Evaluate $A_3$ for $\alpha\in (0,1/c_0)$, $\alpha=1/c_0$, and $\alpha\in (1/c_0,\infty)$, and we have
\begin{align*}
	A_3\leq \begin{dcases}
		\frac{1}{1-c_0\alpha},&\alpha\in (0,1/c_0),\\
		\log (k+h),&\alpha=1/c_0,\\
		\frac{e (k+h)^{c_0\alpha-1}}{c_0\alpha-1},&\alpha\in (1/c_0,\infty).
	\end{dcases}
\end{align*}
It follows that
\begin{align}\label{eq:A_22}
	A_2\leq \begin{dcases}
		\frac{4\alpha^2}{(1-c_0\alpha)(k+h)^{c_0\alpha}},&\alpha\in (0,1/c_0),\\
		\frac{4\alpha^2\log(k+h)}{k+h},&\alpha=1/c_0,\\
		\frac{4e\alpha^2}{(c_0\alpha-1)(k+h)},&\alpha\in (1/c_0,\infty).
	\end{dcases}
\end{align}
Substituting the upper bounds for the terms $A_1$ (cf. Eq. (\ref{eq:A_1})) and $A_2$ (cf. Eq. (\ref{eq:A_22})) into Eq. (\ref{eq:sa-rate}) proves Corollary \ref{thm:diminishing_step_size} (1).

Now consider the case where $\xi\in (0,1)$.
Let $\{u_k\}_{k\geq K}$ be a sequence defined by
\begin{align*}
	u_{k+1}=\left(1-\frac{c_0\alpha}{(k+h)^\xi}\right)u_k+\frac{\alpha^2}{(k+h)^{2\xi}},\quad u_{K}=0.
\end{align*}
It is easy to verify that $u_k=A_2$. We next use induction on $u_k$ to show that 
\begin{align}\label{eq:A_2}
	u_k\leq \frac{2\alpha}{c_0}\frac{1}{(k+h)^\xi}.
\end{align}
Since $u_{K}=0\leq \frac{2\alpha}{c_0}\frac{1}{(K+h)^\xi}$, we have the base case. Now suppose $u_k\leq \frac{2\alpha}{c_0}\frac{1}{(k+h)^\xi}$ for some $k\geq K$. Consider the difference between $\frac{2\alpha}{c_0}\frac{1}{(k+1+h)^\xi}$ and $u_{k+1}$. We have
\begin{align}
	\frac{2\alpha}{c_0}\frac{1}{(k+1+h)^\xi}-u_{k+1}
	\geq\;& \frac{2\alpha}{c_0}\frac{1}{(k+1+h)^\xi}-\left(1-\frac{c_0\alpha}{(k+h)^\xi}\right)\frac{2\alpha}{c_0}\frac{1}{(k+h)^\xi}-\frac{\alpha^2}{(k+h)^{2\xi}}\nonumber\\
	=\;&\frac{2\alpha}{c_0}\frac{1}{(k+h)^{2\xi}}\left[\frac{c_0\alpha}{2}-(k+h)^\xi\left(1-\left(\frac{k+h}{k+1+h}\right)^\xi\right)\right]\nonumber\\
	\geq\;& \frac{2}{c_0\alpha}\frac{1}{(k+h)^{2\xi}}\left[\frac{c_0\alpha}{2}-\frac{\xi}{(k+h)^{1-\xi}}\right]\label{eq:100}\\
	\geq\;& 0\label{eq:101},
\end{align}
where Eq. (\ref{eq:100}) follows from 
\begin{align*}
	\left(\frac{k+h}{k+1+h}\right)^\xi&=\left[\left(1+\frac{1}{k+h}\right)^{k+h}\right]^{-\xi/(k+h)}\\
	&\geq e^{-\xi/(k+h)}\\
	&\geq 1-\frac{\xi}{k+h},
\end{align*}
and Eq. (\ref{eq:101}) follows from $k\geq K\geq [2\xi/(c_0\alpha)]^{1/(1-\xi)}$.

Substituting the upper bounds for the terms $A_1$ (cf. Eq. (\ref{eq:A_1})) and $A_2$ (cf. Eq. (\ref{eq:A_2})) into Eq. (\ref{eq:sa-rate}) proves Corollary \ref{thm:diminishing_step_size} (2).

\subsection{Proof of Lemma \ref{le_1}}\label{pf:le_1}
For all $k\geq t_k$, we have
\begin{align*}
	\mathbb{E}_k[\| F(X_k,\theta_k)+w_k\|^2]
	\leq \;&\mathbb{E}_k[(\| F(X_k,\theta_k)\|+\|w_k\|)^2]\tag{Triangle inequality}\\
	\leq \;&\mathbb{E}_k[(L_1+L_2)^2(\|\theta_k\|+1)^2]\tag{Eq. (\ref{affine_growth_rate-F}) and Assumption \ref{as:markov-chain} (2)}\\
	\leq \;&L^2\mathbb{E}_k[(\|\theta_k-\theta^*\|+\|\theta^*\|+1)^2]\tag{$L=L_1+L_2$ and Triangle inequality}\\ \leq \;&2L^2\mathbb{E}_k[\|\theta_k-\theta^*\|^2+(\|\theta^*\|+1)^2],
\end{align*}
where the last line follows from $(p+q)^2\leq 2(p^2+q^2)$ for any $p,q\in\mathbb{R}$.

\subsection{Proof of Lemma \ref{le:mixing_time}}\label{pf:le:mixing}
We first recall an equivalent formula of computing the total variation distance \cite{charalambous2014extremum}.

\begin{lemma}\label{le:tv}
	Let $\mu_1$ and $\mu_2$ be two probability measures on $(\Omega,\mathcal{F})$. Let $p_1$,$p_2$ be the Radon-Nikodym derivatives of $\mu_1$ and $\mu_2$ w.r.t. some base probability measure $\nu$ (one can choose $\nu=\frac{\mu_1+\mu_2}{2}$). Then,
	\begin{align*}
		\|\mu_1-\mu_2\|_{\text{TV}}=\frac{1}{2}\int_{\Omega}|p_1-p_2|d\nu.
	\end{align*}
\end{lemma}

We next proceed to prove Lemma \ref{le:mixing_time}. For any given state $x\in\mathcal{X}$, let $p^k_x$ and $q_X$ be the Radon-Nikodym derivatives of the probability measures $P^k(x,\cdot)$ and $\mu_X(\cdot)$ with respect to some base probability measure $\nu$. Use the definition of mixing time (cf. Definition \ref{def:mixing_time}) and Lemma \ref{le:tv}, and we have for all $x$, $\theta$, and $k\geq t_\delta$ that
\begin{align*}
	\|\mathbb{E}[F(X_k,\theta)\mid X_0=x]-\bar{F}(\theta)\|
	=\;&\left\|\int_{\mathcal{X}}F(y,\theta)P^k(x,d(y))-\int_{\mathcal{X}}F(y,\theta)\mu_X(dy))\right\|\\
	=\;&\left\|\int_{\mathcal{X}}F(y,\theta)(p_x^k-q_X)d\nu\right\|\\
	\leq \;&\int_{\mathcal{X}}\|F(y,\theta)\||p_x^k-q_X|d\nu\tag{Jensen's inequality}\\
	\leq \;&L_1(\|\theta\|+1)\int_{\mathcal{X}}|p_x^k-q_X|d\nu\tag{Eq. (\ref{affine_growth_rate-F})}\\
	=\;&L_1(\|\theta\|+1)2\|P^k(x,\cdot)-\mu_X(\cdot)\|_{\text{TV}}\tag{Lemma \ref{le:tv}}\\
	\leq \;&2L_1(\|\theta\|+1)C\rho^k\tag{Assumption \ref{as:markov-chain}}\\
	\leq \;&2L_1(\|\theta\|+1)\delta\tag{$k\geq t_\delta$ and Definition \ref{def:mixing_time}}
\end{align*}

\subsection{Proof of Lemma \ref{le_2}}\label{pf:le_2}
Given $k_1<k_2$, we first upper bound $\|\theta_t\|$ for any $t\in[k_1,k_2]$.  Using Eq. (\ref{affine_growth_rate-F}) and Assumption \ref{as:markov-chain} (2), we have
\begin{align}
	\|\theta_{t+1}\|-\|\theta_t\|&\leq\|\theta_{t+1}-\theta_{t}\|\nonumber\\
	&=\alpha_k\|F(X_k,\theta_k)+w_k\|\nonumber\\
	&\leq (L_1+L_2)\alpha_k(\|\theta\|+1)\nonumber\\
	&\leq L\alpha_t(\|\theta_t\|+1),\label{2}
\end{align}
which gives $(\|\theta_{t+1}\|+1)\leq(L\alpha_t+1)(\|\theta_t\|+1)$. Recursively applying the preceding inequality, then using the fact that $1+x\leq e^x$ for all $x\in\mathbb{R}$, we have for all $t\in [k_1,k_2]$:
\begin{align*}
	\|\theta_{t}\|+1&\leq \prod_{j=k_1}^{t-1}(L\alpha_j+1)(\|\theta_{k_1}\|+1)\\
	&\leq \exp(L\alpha_{k_1,k_2-1})(\|\theta_{k_1}\|+1).
\end{align*}
Since $e^x\leq 1+2x$ for all $x\in [0,1/2]$ and $\alpha_{k_1,k_2-1}\leq \frac{1}{4L}$, we further obtain
\begin{align*}
	\|\theta_{t}\|+1\leq (1+2L\alpha_{k_1,k_2-1})(\|\theta_{k_1}\|+1)\leq  2(\|\theta_{k_1}\|+1).
\end{align*}
It follows from the previous inequality and Eq. (\ref{2}) that
\begin{align*}
	\|\theta_{k_2}-\theta_{k_1}\|\leq \sum_{t=k_1}^{k_2-1}\|\theta_{t+1}-\theta_t\|
	\leq 2L\alpha_{k_1,k_2-1}(\|\theta_{k_1}\|+1).
\end{align*}
Since $\alpha_{k_1,k_2-1}\leq \frac{1}{4L}$ and
\begin{align*}
	\|\theta_{k_2}-\theta_{k_1}\|&\leq
	2L\alpha_{k_1,k_2-1}(\|\theta_{k_1}\|+1)\\
	&\leq 2L\alpha_{k_1,k_2-1}(\|\theta_{k_2}-\theta_{k_1}\|+\|\theta_{k_2}\|+1)\\
	&\leq \frac{1}{2}\|\theta_{k_2}-\theta_{k_1}\|+2L\alpha_{k_1,k_2-1}(\|\theta_{k_2}\|+1),
\end{align*}
we have by rearranging terms that 
\begin{align*}
	\|\theta_{k_2}-\theta_{k_1}\|\leq 4L\alpha_{k_1,k_2-1}(\|\theta_{k_2}\|+1).
\end{align*}

\subsection{Proof of Lemma \ref{le_3}}\label{pf:le_3}
We begin by decomposing the following term on the LHS of the desired inequality:
\begin{align*}
	\mathbb{E}_k[(\theta_k-\theta^*)^\top(F(X_k,\theta_k)-\bar{F}(\theta_k))]
	=\;&\underbrace{\mathbb{E}_k[(\theta_k-\theta_{k-t_k})^\top(F(X_k,\theta_k)-\bar{F}(\theta_k))]}_{(T_1)}\\
	&+\underbrace{(\theta_{k-t_k}-\theta^*)^\top(\mathbb{E}_k[F(X_k,\theta_{k-t_k})]-\bar{F}(\theta_{k-t_k}))}_{(T_2)}\\
	&+\underbrace{\mathbb{E}_k[(\theta_{k-t_k}-\theta^*)^\top(F(X_k,\theta_k)-F(X_k,\theta_{k-t_k}))}_{T_3}\\
	&+\underbrace{\mathbb{E}_k[(\theta_{k-t_k}-\theta^*)^\top(\bar{F}(\theta_{k-t_k})-\bar{F}(\theta_k))]}_{(T_4)}.
\end{align*}
Consider the term $(T_1)$. Since $\alpha_{k-t_k,k-1}\leq \frac{1}{4L}$, Lemma \ref{le_2} is applicable for $k_1=k-t_k$ and $k_2=k$. Therefore, we have:
\begin{align}
	(T_1)
	=\;&\mathbb{E}_k[(\theta_k-\theta_{k-t_k})^\top(F(X_k,\theta_k)-\bar{F}(\theta_k))]\nonumber\\
	\leq\; &\mathbb{E}_k[\|\theta_k-\theta_{k-t_k}\|\|F(X_k,\theta_k)-\bar{F}(\theta_k)\|]\nonumber\\
	\leq\; &\mathbb{E}_k[\|\theta_k-\theta_{k-t_k}\|(\|F(X_k,\theta_k)\|+\|\bar{F}(\theta_k)\|)]\nonumber\\
	\leq\; &2L\mathbb{E}_k[\|\theta_k-\theta_{k-t_k}\|(\|\theta_k\|+1)]\nonumber\\
	\leq\;&8L^2\alpha_{k-t_k,k-1}\mathbb{E}_k\left[(\|\theta_k\|+1)^2\right]\tag{Lemma \ref{le_2}}\\
	\leq\;&8L^2\alpha_{k-t_k,k-1}\mathbb{E}_k\left[(\|\theta_k-\theta^*\|+\|\theta^*\|+1)^2\right]\nonumber\\
	\leq\; &16L^2\alpha_{k-t_k,k-1}(\mathbb{E}_k\left[\|\theta_k-\theta^*\|^2\right]+(\|\theta^*\|+1)^2)\label{31}.
\end{align}	
Now consider the term ($T_2$). Since Lemma \ref{le:mixing_time} implies that
\begin{align*}
	\|\mathbb{E}_k[F(X_k,\theta_{k-t_k})]-\bar{F}(\theta_{k-t_k})\|
	\leq 2\alpha_kL (\|\theta_{k-t_k}\|+1),
\end{align*}
we have by Cauchy-Schwarz inequality that
\begin{align*}
	(T_2)
	\leq2\alpha_kL\mathbb{E}_k[(\|\theta_{k-t_k}\|+1)\|\theta_{k-t_k}-\theta^*\|].
\end{align*}
To further control $(T_2)$, using Lemma \ref{le_2} together with our assumption that $\alpha_{k-t_k,k-1}\leq \frac{1}{4L}$, we have
\begin{align}\label{eq:200}
	\|\theta_k-\theta_{k-t_k}\|\leq 4L\alpha_{k-t_k,k-1}(\|\theta_k\|+1)\leq \|\theta_k\|+1.
\end{align}
Therefore, we have
\begin{align*}
	&\|\theta_{k-t_k}-\theta^*\|(\|\theta_{k-t_k}\|+1)\\
	\leq\; &(\|\theta_k-\theta_{k-t_k}\|+\|\theta_k-\theta^*\|)(\|\theta_k-\theta_{k-t_k}\|+\|\theta_k-\theta^*\|+\|\theta^*\|+1)\\
	\leq\;& (\|\theta_k\|+\|\theta_k-\theta^*\|+1)(\|\theta_k\|+\|\theta_k-\theta^*\|+\|\theta^*\|+2)\\
	\leq\; &(2\|\theta_k-\theta^*\|+\|\theta^*\|+1)(2\|\theta_k-\theta^*\|+2\|\theta^*\|+2)\\
	\leq \;&4(\|\theta_k-\theta^*\|+\|\theta^*\|+1)^2\\
	\leq \;&8[\|\theta_k-\theta^*\|^2+(\|\theta^*\|+1)^2].
\end{align*}
It follows that
\begin{align}\label{eq:T21}
	(T_2) \leq 16\alpha_kL(\mathbb{E}_k[\|\theta_k-\theta^*\|^2]+(\|\theta^*\|+1)^2).
\end{align}
We next bound the terms $(T_3)$ and $(T_4)$. Using Assumption \ref{as:Lipschitz}, we have
\begin{align}
	(T_3)+(T_4)
	\leq \;&2L\|\theta_{k-t_k}-\theta^*\|\mathbb{E}_k[\|\theta_k-\theta_{k-t_k}\|]\nonumber\\
	\leq\;& 8L^2\alpha_{k-t_k,k-1}\mathbb{E}_k[\|\theta_{k-t_k}-\theta^*\|(\|\theta_{k}\|+1)]\tag{Lemma \ref{le_2}}\\
	\leq\;& 8L^2\alpha_{k-t_k,k-1}\mathbb{E}_k[(\|\theta_k-\theta_{k-t_k}\|+\|\theta_{k}-\theta^*\|)(\|\theta_k\|+1)]\nonumber\\
	\leq\;& 8L^2\alpha_{k-t_k,k-1}\mathbb{E}_k[(\|\theta_k\|+\|\theta_{k}-\theta^*\|+1)(\|\theta_k-\theta^*\|+\|\theta^*\|+1)]\tag{Eq. (\ref{eq:200})}\\
	\leq\;& 8L^2\alpha_{k-t_k,k-1}\mathbb{E}_k[(2\|\theta_{k}-\theta^*\|+\|\theta^*\|+1)(\|\theta_k-\theta^*\|+\|\theta^*\|+1)]\nonumber\\
	\leq\;& 16L^2\alpha_{k-t_k,k-1}\mathbb{E}_k[(\|\theta_k-\theta^*\|+\|\theta^*\|+1)^2]\nonumber\\
	\leq\;& 32L^2\alpha_{k-t_k,k-1}\mathbb{E}_k[\|\theta_k-\theta^*\|^2+(\|\theta^*\|+1)^2].\label{eq:T22}
\end{align}
Finally, combining the upper bounds we derived for the terms $\{(T_i)\}_{1\leq i\leq 4}$ in Eqs. (\ref{31}), (\ref{eq:T21}), and (\ref{eq:T22}), we have
\begin{align*}
	(c)=\;&2\alpha_k((T_1)+(T_2)+(T_3)+(T_4))\\
	\leq\;& 128L^2\alpha_k\alpha_{k-t_k,k-1}\left[\mathbb{E}_k[\|\theta_k-\theta^*\|^2]+(\|\theta^*\|+1)^2\right],
\end{align*}
where the last line follows from $L\geq 1$ and $\alpha_k\leq \alpha_{k-t_k,k-1}$.

\subsection{Proof of Lemma \ref{le:recursion}}\label{pf:le:recursion}
Substituting the upper bounds we obtained for the terms $(a)-(d)$ into Eq. (\ref{eq:expand}), we have
\begin{align*}
	\mathbb{E}_k[\|\theta_{k+1}-\theta^*\|^2]-\mathbb{E}_k[\|\theta_k-\theta^*\|^2]
	\leq\;&(-2c_0\alpha_k+128L^2\alpha_k\alpha_{k-t_k,k-1}+2L_1^2\alpha_k^2)\mathbb{E}_k[\|\theta_k-\theta^*\|^2]\\
	&+(128L^2\alpha_k\alpha_{k-t_k,k-1}+2L_1^2\alpha_k^2)(\|\theta^*\|+1)^2\\
	\leq\;&(-2c_0\alpha_k+130L^2\alpha_k\alpha_{k-t_k,k-1})\mathbb{E}_k[\|\theta_k-\theta^*\|^2]\\
	&+130L^2\alpha_k\alpha_{k-t_k,k-1}(\|\theta^*\|+1)^2.
\end{align*}
The result then follows by taking the total expectation on both sides of the previous inequality.}

\section{Proof of All Technical Results in Section \ref{sec:RL}}

\subsection{Proof of Proposition \ref{thm:apply-to-q-learning}}\label{pf:thm:apply-to-q-learning}
\begin{enumerate}[(1)]
	\item We first show that $\{X_k\}$ is irreducible and aperiodic. Consider two arbitrary states $x_1=(s_1,a_1,s_1')$ and $ x_2=(s_2,a_2,s_2')\in \mathcal{X}$. Since $\{S_k\}$ is irreducible, there exists $k>0$ such that $P^k(s_1',s_2)>0$. Hence we have $P^{k+1}(x_1,x_2)=P^k(s_1',s_2)\pi(a_2|s_2)P_{a_2}(s_2,s_2')>0$. It follows that $\{X_k\}$ is irreducible. 
	
	To show $\{X_k\}$ is aperiodic, assume for a contradiction that $\{X_k\}$ is periodic with period $T\geq 2$. Since $\{X_k\}$ is irreducible, every state in $\mathcal{X}$ has the same period. Therefore, for any $x=(s,a,s')\in \mathcal{X}$, we must have $P^k(x,x)=0$ for all $k$ not divisible by $T$. However, we have for any $k$ not divisible by $T$:
	\begin{align}
		P^k(s',s')&=\sum_{s \in \mathcal{S}}P^{k-1}(s',s)P(s,s')\nonumber\\
		&=\sum_{s \in \mathcal{S}}\sum_{a \in \mathcal{A}}P^{k-1}(s',s)\pi(a|s)P_a(s,s')\nonumber\\
		&=\sum_{s \in \mathcal{S}}\sum_{a \in \mathcal{A}}P^k((s,a,s'),(s,a,s'))\label{aperiodic}\\
		&=0.\label{aperiodic1}
	\end{align}
	To see (\ref{aperiodic}), since $\{S_k\}$ is a Markov chain, we have
	\begin{align*}
		P^k((s,a,s'),(s,a,s'))
		=\;&\mathbb{P}(s_k=s,a_k=a,s_{k+1}=s'|s_0=s,a_0=a,s_1=s')\\
		=\;&\mathbb{P}(s_k=s,a_k=a,s_{k+1}=s'|s_1=s')\\
		=\;&P^{k-1}(s',s)\pi(a|s)P_a(s,s').
	\end{align*}
	Therefore, Eq. (\ref{aperiodic1}) shows that the period of $s'$ is at least $T$, which is a contradiction. Hence the Markov chain $\{X_k\}$ must be aperiodic.
	
	Since $\{X_k\}$ is an irreducible and aperiodic Markov chain with a finite state-space $\mathcal{X}$, it admits a unique stationary distribution, which we have denoted by $\mu_X$. In view of the definition of $X_k$, it is clear that $\mu_X(s,a,s')=\mu_S(s)\pi(a|s)P_a(s,s')$ for all $(s,a,s')$. Therefore, for any $x\in\mathcal{X}$, we have by definition of the total variation distance that
		\begin{align*}
			\|P^{k+1}(x,\cdot)-\mu_X(\cdot)\|_{\text{TV}}
			=\;&\frac{1}{2}\sum_{\tilde{x}\in\mathcal{X}}\left|P^{k+1}(x,\tilde{x})-\mu_X(\tilde{x})\right|\\
			=\;&\frac{1}{2}\sum_{(\tilde{s},\tilde{a},\tilde{s}')\in\mathcal{X}}\left|P_\pi^k(s',\tilde{s})-\mu_S(\tilde{s})\right|\pi(\tilde{a}|\tilde{s})P_{\tilde{a}}(\tilde{s},\tilde{s}')\\
			=\;&\frac{1}{2}\sum_{(\tilde{s},\tilde{a},\tilde{s}')\in\mathcal{X}}\left|P_\pi^k(s',\tilde{s})-\mu_S(\tilde{s})\right|\\
			\leq \;&\|P^k(\tilde{s},\cdot)-\mu_S(\cdot)\|_{\text{TV}}\\
			\leq \;&C'\rho'^k
		\end{align*}
		for all $k\geq 0$. It follows that $\max_{x\in\mathcal{X}}\|P^{k+1}(x,\cdot)-\mu_X(\cdot)\|_{\text{TV}}\leq C'\rho'^k$ for all $k\geq 0$.
	
	\item Using Cauchy-Schwarz inequality, and our assumption that $\|\phi(s,a)\|=1$ for all state-action pairs, we have for any $\theta_1,\theta_2$ and $x=(s,a,s')$:
	\begin{align*}
		\|F(x,\theta_1)-F(x,\theta_2)\|
		= \;&\|\phi(s,a)(\mathcal{R}(s,a)+\gamma\max_{a_1 \in \mathcal{A}}\phi(s',a_1)^\top\theta_1-\phi(s,a)^\top\theta_1)\\
		&-\phi(s,a)(\mathcal{R}(s,a)+\gamma\max_{a_2 \in \mathcal{A}}\phi(s',a_2)^\top\theta_2-\phi(s,a)^\top\theta_2)\|\\
		\leq\; &\gamma\|\phi(s,a)(\max_{a_1 \in \mathcal{A}}\phi(s',a_1)^\top\theta_1-\max_{a_2 \in \mathcal{A}}\phi(s',a_2)^\top\theta_2)\|+\|\phi(s,a)\phi(s,a)^\top(\theta_1-\theta_2)\|\\
		\leq \;&\gamma|\max_{a_1 \in \mathcal{A}}\phi(s',a_1)^\top\theta_1-\max_{a_2 \in \mathcal{A}}\phi(s',a_2)^\top\theta_2|+\|\theta_1-\theta_2\|.
	\end{align*}
	Since
	\begin{align}
		|\max_{a_1 \in \mathcal{A}}\phi(s',a_1)^\top\theta_1-\max_{a_2 \in \mathcal{A}}\phi(s',a_1)^\top\theta_2|
		\leq \;&\max_{a' \in \mathcal{A}}|\phi(s',a')^\top(\theta_1-\theta_2)|\label{eq:70}\\
		\leq\;& \max_{a' \in \mathcal{A}}\|\phi(s',a')\|\|\theta_1-\theta_2\|\nonumber\\
		\leq\;& \|\theta_1-\theta_2\|\nonumber,
	\end{align}
	we have for any $\theta_1,\theta_2$ and $x$:
	\begin{align*}
		\|F(x,\theta_1)-F(x,\theta_2)\|
		\leq(\gamma+1)\|\theta_1-\theta_2\|
		\leq M\|\theta_1-\theta_2\|.
	\end{align*}
	Moreover, we have 
	\begin{align*}
		\|F(x,0)\|=\|\phi(s,a)\mathcal{R}(s,a)\|\leq r_{\max}<M
	\end{align*}
	for any $x\in\mathcal{X}$.
	\item 
	Using the fact that $\bar{F}(\theta^*)=0$, we have
	\begin{align}
		(\theta-\theta^*)^\top (\bar{F}(\theta)-\bar{F}(\theta^*))
		=\;&\gamma (\theta-\theta^*)^\top \mathbb{E}_{\mu_S}[\phi(S,A)(\max_{a_1 \in \mathcal{A}}\phi(S',a_1)^\top\theta-\max_{a_2 \in \mathcal{A}}\phi(S',a_2)^\top\theta^*)]\nonumber\\
		&-\mathbb{E}_{\mu_S}[(\phi(S,A)^\top(\theta-\theta^*))^2]\nonumber\\
		\leq \;&\gamma\mathbb{E}_{\mu_S}[|\phi(S,A)^\top (\theta-\theta^*)|\max_{a' \in \mathcal{A}}|\phi(S',a')^\top(\theta-\theta^*)|]\nonumber\\
		&-\mathbb{E}_{\mu_S}[(\phi(S,A)^\top(\theta-\theta^*))^2]\label{eq:51}\\
		\leq \;&\gamma\sqrt{\mathbb{E}_{\mu_S}[(\phi(S,A)^\top (\theta-\theta^*))^2]}\sqrt{\mathbb{E}_{\mu_S}[\max_{a \in \mathcal{A}}(\phi(S,a)^\top(\theta-\theta^*))^2]}\nonumber\\
		&-\mathbb{E}_{\mu_S}[(\phi(S,A)^\top(\theta-\theta^*))^2].\label{eq:52}
	\end{align}
	Eq. (\ref{eq:51}) follows from Eq. (\ref{eq:70}). Eq. (\ref{eq:52}) follows from the fact that when $S\sim \mu_S$, we have $S'\sim\mu_S$. For simplicity of notation, denote 
	\begin{align*}
		A=\sqrt{\mathbb{E}_{\mu_S}[(\phi(S,A)^\top (\theta-\theta^*))^2]},\text{ and }B=\sqrt{\mathbb{E}_{\mu_S}[\max_{a \in \mathcal{A}}(\phi(S,a)^\top(\theta-\theta^*))^2]}.
	\end{align*} 
	Since Assumption \ref{as:q-behhavior-policy} gives $\gamma^2B^2-A^2\leq -\kappa \|\theta-\theta^*\|^2$, we have
	\begin{align*}
		(\theta-\theta^*)^\top \bar{F}(\theta)
		\leq\frac{\gamma^2 B^2-A^2}{\gamma B/A+1}\leq -\frac{\kappa}{2}\|\theta-\theta^*\|^2.
	\end{align*}
\end{enumerate}

\subsection{Proof of Proposition \ref{thm:1d-iff}}\label{subsec:pf:thm:1d-iff}

We first show that Eq. (\ref{1d-weaker}) implies $h^+<0$, and $h^-<0$. Note that Jensen's inequality implies
\begin{align}
	\mathbb{E}_{\mu_S}[\max_{a' \in \mathcal{A}}\phi(S,a')^2]
	=\;&\mathbb{E}_{\mu_S}\left\{\max\left[(\max_{a' \in \mathcal{A}}\phi(S,a'))^2 , (\min_{a' \in \mathcal{A}}\phi(S,a'))^2\right]\right\}\nonumber\\
	\geq\;& \max\left\{\mathbb{E}_{\mu_S}[(\max_{a' \in \mathcal{A}}\phi(S,a'))^2], \mathbb{E}_{\mu_S}[(\min_{a' \in \mathcal{A}}\phi(S,a'))^2]\right\}.\label{weaker-to-H}
\end{align}
Thus, using Eq. (\ref{1d-weaker}), we have
\begin{align*}
	h^+
	=\;&\mathbb{E}_{\mu_S}[\gamma\phi(S,A)\max_{a' \in \mathcal{A}}\phi(S',a')]-\mathbb{E}_{\mu_S}[\phi(S,A)^2]\\
	=\;& \mathbb{E}_{\mu_S}[\gamma\phi(S,A)\max_{a' \in \mathcal{A}}\phi(S',a')]-\sqrt{\mathbb{E}_{\mu_S}[\phi(S,A)^2]\mathbb{E}_{\mu_S}[\phi(S,A)^2]}\\
	<\;& \mathbb{E}_{\mu_S}[\gamma\phi(S,A)\max_{a' \in \mathcal{A}}\phi(S',a')]-\gamma\sqrt{\mathbb{E}_{\mu_S}[\max_{a' \in \mathcal{A}}\phi(S,a')^2]\mathbb{E}_{\mu_S}[\phi(S,A)^2]}\\
	\leq\;& 0,
\end{align*}
where the last inequality follows from Cauchy-Schwarz inquality and the fact that $S'$ and $S$ are equal in distribution if $S\sim\mu_S$. Similarly, we also have $h^-<0$. 

We next prove the equivalence stated in Proposition \ref{thm:1d-iff}. By definition of $h^+$ and $h^-$, in uni-dimensional case, ODE (\ref{ODE-q-learning}) can be equivalently written as
\begin{align*}
	\dot{\theta}(t)=\begin{cases}
		h^+ \theta(t)+r_\pi,&\theta(t)\geq 0,\\
		h^- \theta(t)+r_\pi,&\theta(t)<0.
	\end{cases}
\end{align*}
In the case where $r_\pi=0$, it is easy to see that ODE (\ref{ODE-q-learning}) is globally asymptotically stable if and only if $h^+,h^-<0$. Now we assume without loss of generality that $r_\pi>0$. The proof for the other case is entirely similar.

\textbf{Sufficiency:} We first note that $\theta^*=-r_\pi/h^+>0$. Let $W(\theta)=\frac{1}{2}(\theta-\theta^*)^2$ be a candidate Lyapunov function. It is clear that $W(\theta)\geq 0$ for all $\theta\in \mathbb{R}$, and $W(\theta)=0$ if and only if $\theta=\theta^*$. Moreover, we have \begin{align*}
	\dot{W}(\theta(t))&=(\theta(t)-\theta^*)\dot{\theta}(t)\\
	&=\begin{cases}
		h^+(\theta(t)-\theta^*)^2,&\theta(t)\geq 0\\
		(\theta(t)-\theta^*)(h^-\theta(t)-h^+\theta^*),&\theta(t)<0.
	\end{cases}
\end{align*}
It is clear that $\dot{W}(\theta(t))<0$ when $\theta(t)\in [0,\theta^*)\cup(\theta^*,\infty)$. For $\theta(t)<0$, since $\theta(t)-\theta^*<0$, $h^+\theta^*=-r_\pi<0$, and $h^-\theta(t)\geq 0$, we must also have $\dot{W}(\theta(t))<0$. Therefore, the time derivative of the Lyapunov function $W(\theta)$ along the trajectory of ODE (\ref{ODE-q-learning}) is strictly negative when $\theta(t)\neq \theta^*$. It then follows from the Lyapunov stability theorem \cite{haddad2011nonlinear,khalil2002nonlinear} that $\theta^*$ is globally asymptotically stable.

\textbf{Necessity:} We prove by contradiction. Suppose that the equilibrium point $\theta^*$ is globally asymptotically stable, but $h^+\geq 0$ or $h^->0$. Suppose that $h^+\geq 0$. When $\theta(0)>\max(0,\theta^*)$, we have $\dot{\theta}(t)=h^+\theta(t)+r_{\pi}\geq r_{\pi}>0$. It follows that $\theta(t)>\theta(0)>\theta^*$ for all $t\geq 0$, which contradict to the fact that $\theta^*$ is a globally asymptotically stable equilibrium point. Suppose that $h^->0$. When $\theta(0)<\min(\theta^*,-(1+r_{\pi})/h^-)$, we have $\dot{\theta}(t)=h^-\theta(t)+r_\pi\leq -1<0$. It follows that $\theta(t)<\theta(0)<\theta^*$ for all $t\geq 0$, which also contradict to the fact $\theta^*$ being globally asymptotically stable.

\subsection{Proof of Proposition \ref{thm:full-dimension}}\label{pf:thm:full-dimension}
When $d=mn$, the feature matrix $\Phi$ is a square matrix. Define 
\begin{align*}
	\Theta_{s,a}=\text{span}\left(\left\{\phi(s',a')| (s',a')\in\mathcal{S}\times\mathcal{A},\;(s',a')\neq (s,a)\right\}\right)^{\perp}.
\end{align*}
Note that $\Theta_{s,a}$ exists for all state-action pairs since $\Phi$ is full rank. Now for a given state-action pair $(s,a)$, let $\theta\neq 0$ be in $\Theta_{s,a}$, Eq. (\ref{con}) implies \begin{align*}
	\gamma^2\mu_S(s)(\phi(s,a)^\top\theta)^2<\mu_S(s)\pi(a|s)(\phi(s,a)^\top\theta)^2,
\end{align*}
which further gives $\gamma^2<\pi(a|s)$. Therefore, by running $(s,a)$ though all state-action pairs, we have $\gamma^2<\min_{(s,a)\in \mathcal{S}\times\mathcal{A}}\pi(a|s)\leq \frac{1}{m}$. Thus, if $\gamma^2\geq 1/m$, there is no behavior policy $\pi$ that satisfies condition (\ref{con}).

\subsection{Computing $\omega(\pi)$}\label{computing_omega}
We here present one way to compute $\omega(\pi)$ for an MDP with a chosen policy $\pi$ when the underlying model is known. Before that, the following definitions are needed.
\begin{definition}
	Let $D\in \mathbb{R}^{mn\times mn}$ be a diagonal matrix with diagonal entries $\{\mu_S(s)\pi(a|s)\}_{(s,a)\in \mathcal{S}\times\mathcal{A}}$, and let $\Sigma=\Phi^\top D\Phi\in \mathbb{R}^{d\times d}$, where $\Phi\in \mathbb{R}^{mn\times d}$ is the feature matrix.
\end{definition}
\begin{definition}
	Let $\mathcal{B}=\mathcal{A}^n\subseteq \mathbb{R}^n$ be the set of all deterministic policies.
\end{definition}
\begin{definition}
	Let $H\in \mathbb{R}^{n\times n}$ be a diagonal matrix with diagonal entries $\{\mu_S(s)\}_{ s \in \mathcal{S}}$, and let $\Sigma_{b}=\Phi_b^\top H\Phi_b\in\mathbb{R}^{d\times d}$, where $\Phi_b\in \mathbb{R}^{n\times d}$ ($b\in \mathcal{B}$) is defined by:
	\begin{align*}
		\Phi_b =\left[\begin{array}{ccc}
			\mbox{---}   & \phi(s_{1},b)^\top & \mbox{---}\\
			...    & ... & ...\\
			\mbox{---}   & \phi(s_{n},b)^\top & \mbox{---}
		\end{array}\right].
	\end{align*}
\end{definition}
We now compute $\omega(\pi)$ given in the following lemma. Let $\lambda_{\max}(\cdot)$ return the largest eigenvalue of a positive semi-definite matrix
\begin{lemma}\label{numerical}
	$\omega(\pi)=\min_{b\in \mathcal{B}}\left[1/\lambda_{\max}(\Sigma^{-1/2}\Sigma_{b}\Sigma^{-1/2})\right]$.
\end{lemma}
\textit{Proof of Lemma \ref{numerical}}:
Recall our definition for $\omega(\pi)$:
\begin{align}\label{key3}
	\omega(\pi)=\min_{\theta\neq 0}\frac{\sum_{s \in \mathcal{S}}\mu_S(s)\sum_{a \in \mathcal{A}}\pi(a|s)(\phi(s,a)^\top\theta)^2}{\sum_{s \in \mathcal{S}}\mu_S(s)\max_{a \in \mathcal{A}}(\phi(s,a)^\top\theta)^2}\cdot
\end{align}
Let $f(\theta)$ be the numerator. Then we have
\begin{align*}
	f(\theta)&=\sum_{s \in \mathcal{S}}\mu_S(s)\sum_{a \in \mathcal{A}}\pi(a|s)(\phi(s,a)^\top\theta)^2\\
	&=\theta^\top \Phi^\top D\Phi\theta=\theta^\top \Sigma\theta.
\end{align*}
Since the diagonal entries of $D$ are all positive, and $\Phi$ is full column rank, the matrix $\Sigma$ is symmetric and positive definite. To represent the denominator of (\ref{key3}) in a similar form, let
\begin{align*}
	g(\theta,b)&=\sum_{s}\mu_S(s)(\phi(s,b)^\top\theta)^2\\
	&=\theta^\top \Phi^\top_bH\Phi_b\theta=\theta^\top\Sigma_{b}\theta,
\end{align*}
where $b\in\mathcal{B}$.
Since the columns of $\Phi_b$ can be dependent, the matrix $\Sigma_{b}$ is in general only symmetric and positive semi-definite. Using the definition of $f(\theta)$ and $g(\theta,b)$, we can rewrite $\omega(\pi)$ as
\begin{align*}
	\omega(\pi)&=\min_{ \theta\neq 0}\frac{f(\theta)}{\max_{b\in \mathcal{B}}g(\theta,b)}\\
	&=\min_{ \theta\neq 0}\min_{b\in \mathcal{B}}\frac{f(\theta)}{g(\theta,b)}\\
	&=\min_{b\in \mathcal{B}}\min_{ \theta\neq 0}\frac{f(\theta)}{g(\theta,b)}.
\end{align*}
Now since $\Sigma$ is positive definite, $\Sigma^{1/2}$ and $\Sigma^{-1/2}$ are both well-defined and positive definite, we have
\begin{align*}
	\min_{ \theta\neq 0}\frac{f(\theta)}{g(\theta,b)}&=\left[\max_{ \theta\neq 0}\frac{g(\theta,b)}{f(\theta)}\right]^{-1}\\
	&=\left[\max_{\theta\neq 0}\frac{\theta^\top\Sigma_{\mu,b}\theta}{\theta^\top\Sigma_{\mu,\pi}\theta}\right]^{-1}\\
	&=\left[\left(\max_{x\neq 0}\frac{\|\Sigma_{\mu,b}^{1/2}\Sigma_{\mu,\pi}^{-1/2}x\|}{\|x\|}\right)^2\right]^{-1}\\
	&=\frac{1}{\lambda_{\max}(\Sigma_{\mu,\pi}^{-1/2}\Sigma_{\mu,b}\Sigma_{\mu,\pi}^{-1/2})}.
\end{align*}
It follows that 
\begin{align*}
	\omega(\pi)=\min_{b\in\mathcal{B}}[1/\lambda_{\max}(\Sigma^{-1/2}\Sigma_{b}\Sigma^{-1/2})].
\end{align*}
\hfill $\square$

\section{On the Existence of Solution to Equation (\ref{eq:pbj})}\label{ap:no_solution}

In this section, we construct an example to show that Eq. (\ref{eq:pbj}) for $Q$-learning with linear function approximation may not admit a solution. Consider an MDP with states-space $\mathcal{S}=\{1,2\}$, action-space $\mathcal{A}=\{1,2\}$, transition probability matrices
\begin{align*}
	P_1=\begin{bmatrix}
		1&0\\
		1&0
	\end{bmatrix}\quad P_2=\begin{bmatrix}
		0&1\\
		0&1
	\end{bmatrix},
\end{align*}
reward function
\begin{align*}
	R=\begin{bmatrix}
		R(1,1)\\
		R(1,2)\\
		R(2,1)\\
		R(2,2)
	\end{bmatrix}=\begin{bmatrix}
		1\\
		2\\
		2\\
		4
	\end{bmatrix},
\end{align*}
and a tunable discount factor $\gamma\in (0,1)$. Let the feature matrix be defined by
\begin{align*}
	\Phi=\begin{bmatrix}
		\phi(1,1)\\
		\phi(1,2)\\
		\phi(2,1)\\
		\phi(2,2)
	\end{bmatrix}=\begin{bmatrix}
		1\\
		2\\
		2\\
		4
	\end{bmatrix}.
\end{align*}
We use a uniform behavior policy, i.e., $\pi(1|1)=\pi(2|1)=\pi(1|2)=\pi(2|2)=0.5$. Then the transition probability matrix $P_\pi$ under policy $\pi$ is given by
\begin{align*}
	P_{\pi}=0.5P_1+0.5P_2=\begin{bmatrix}
		0.5&0.5\\
		0.5&0.5
	\end{bmatrix},
\end{align*}
and the unique stationary distribution $\mu_S$ of the Markov chain $\{S_k\}$ under policy $\pi$ is given by $\mu_S(1)=\mu_S(2)=0.5$. 

Consider the target equation (\ref{eq:pbj}). In this example, after straightforward calculation, Eq. (\ref{eq:pbj}) reduces to
\begin{align*}
	\theta=\begin{dcases}
		1+\gamma\frac{6}{5}\theta,&\theta\geq 0,\\
		1+\gamma \frac{3}{5} \theta,&\theta<0,
	\end{dcases}
\end{align*}
which has no solution when $\gamma\in (5/6,1)$.

\section{Numerical Simulations}\label{ap:larger}
To complement the numerical experiments presented in Section \ref{subsec:RL:experiments}, here we implement the 
$Q$-learning with linear function approximation algorithm on a larger MDP. 
We first introduce our experimental setup and then state our results.  

\subsection{Setup}
We consider an MDP with $100$ states and $10$ actions, where rewards and transition probabilities are generated as follows:

\textbf{Rewards. } The reward $\mathcal{R}(s,a)$ for each state-action pair $(s,a)$ is drawn uniformly on $[0,1]$.

\textbf{Transition probabilities. } For each state-action pair $(s,a) \in \mathcal{S} \times \mathcal{A}$, the probabilities $P_a(s,s')$ of each successor state $s' \in \mathcal{S}$ are chosen as random partitions of the unit interval. That is,  $99$ numbers are chosen uniformly randomly between $0$ and $1$, dividing that interval into $100$ numbers that sum to one -- the probabilities of the $100$ successor states.

Moreover, we consider a feature matrix $\Phi$ with $100$ features (recall that there are total $1000$ state-action pairs) for each state-action pair $(s,a) \in \mathcal{S} \times \mathcal{A}$, where each element is drawn from the Bernoulli distribution with success probability $p=0.5$. We repeat this process until we obtain a full column rank feature matrix $\Phi$. We further normalize the features to ensure $\Vert \phi(s,a) \Vert \leq 1$ for all $(s,a) \in \mathcal{S} \times \mathcal{A}$. Furthermore, the behavior policy $\pi$ is chosen to take each action with equal probability in each state $s \in \mathcal{S}$.

\subsection{Results}
In our first set of experiments, we choose constant stepsize $\alpha = 0.01$ and discount factor $\gamma \in \{0.5, 0.7, 0.9\}$. In Fig. \ref{fig:5}, we plot $\Vert \Phi \theta_k - Q^*\Vert$ as a function of the iteration $k$, where $Q^*$ associated with each $\gamma$ is the optimal $Q$ value function computed by the value iteration algorithm. Here, $\Phi \theta_k$ converges when $\gamma \in \{0.5, 0.7\}$, but diverges when $\gamma = 0.9$. This again shows that the algorithm is likely to diverge when $\gamma$ is close to $1$ and that the condition \eqref{key} is sufficient but not necessary for convergence. To demonstrate the exponential convergence rate for constant stepsize, we plot $\log \mathbb{E} \left[\Vert \theta_k - \theta^* \Vert^2\right]$ as the function of the iteration $k$ when $\gamma = 0.5$, where $\theta^*$ is the solution of the projected Bellman equation (\ref{eq:pbj}), estimated by the projected value iteration algorithm. Note that, we repeat running the algorithm for $1000$ times and use the average as an approximation to the expectation. In Fig. \ref{fig:6}, we observe that the graph is nearly a straight line when $k$ is large enough, meaning that $\theta_k\rightarrow\theta^*$ geometrically fast, which agrees with Theorem \ref{thm:q-constant} (1).

\begin{figure}[h]
	\centering
	\begin{minipage}{.40\textwidth}
		\centering
		\includegraphics[width=\linewidth]{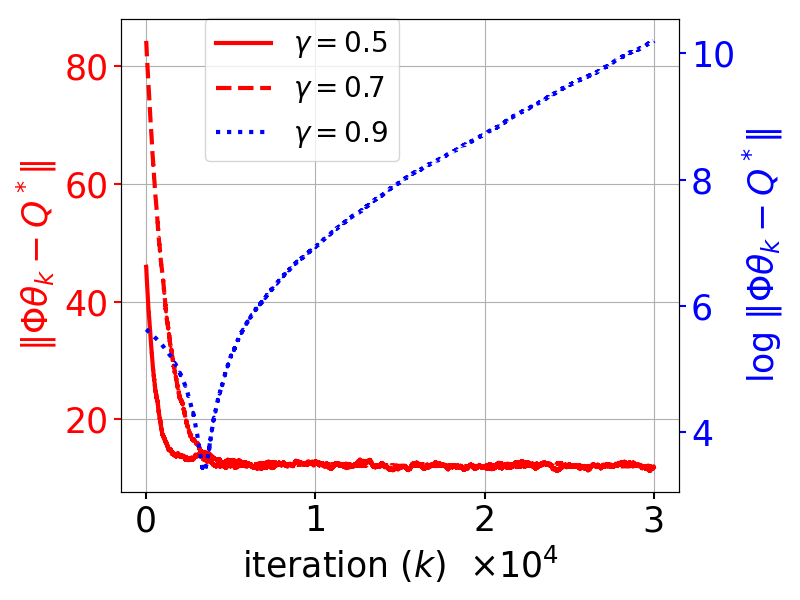}
		\caption{{\small Convergence of $Q$-learning with linear function approximation for discount factor $\gamma \in \{0.5, 0.7, 0.9\}$.}}
		\label{fig:5}
	\end{minipage}%
\hfill
	\begin{minipage}{.40\textwidth}
		\centering
		\includegraphics[width=\linewidth]{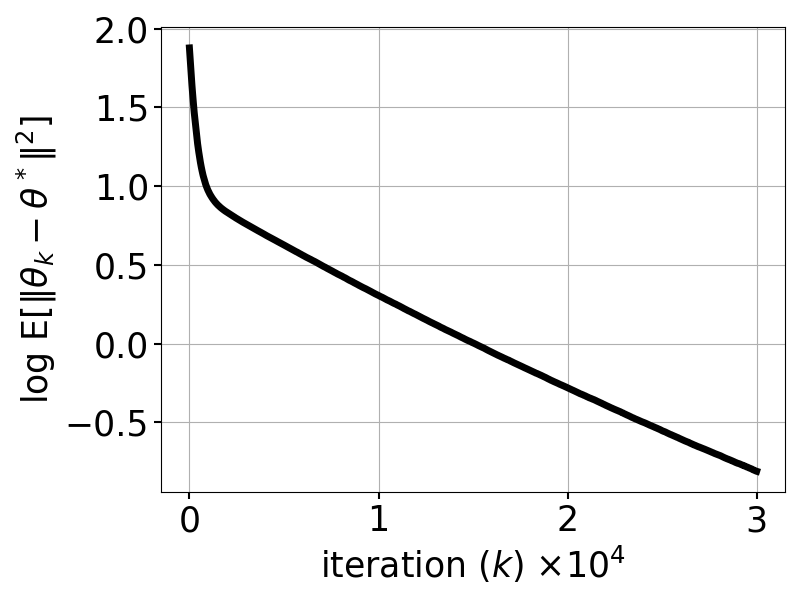}
		\caption{{\small Exponential convergence rate of $Q$-learning with linear function approximation for $\gamma=0.5$}}
		\label{fig:6}
	\end{minipage}
\end{figure}

In our second set of experiments, we consider diminishing stepsizes $\alpha_k = \frac{\alpha}{k^\xi}$, where $\xi \in \{0.4, 0.6, 0.8, 1.0\}$. In the case where $\xi = 1$, the constant $\alpha$ is chosen such that $\kappa \alpha > 2$ to achieve the optimal convergence rate. In addition, the discount factor $\gamma$ is set to be $0.5$. Fig. \ref{fig:7} shows that the algorithm converges for all $\xi \in \{0.4, 0.6, 0.8, 1.0\}$ and the algorithm converges faster with larger $\xi$. To further illustrate the rate of convergence for each choice of $\xi$, we plot $\log \mathbb{E} \left[\Vert \theta_k - \theta^* \Vert^2\right]$ as a function of $\log k$ in Fig. \ref{fig:8} and focus on its asymptotic behavior. We can observe that the slope is $-\xi$, which agrees with Corollary \ref{thm:diminishing_step_size}.

\begin{figure}[h]
	\centering
	\begin{minipage}{.40\textwidth}
		\centering
		\includegraphics[width=\linewidth]{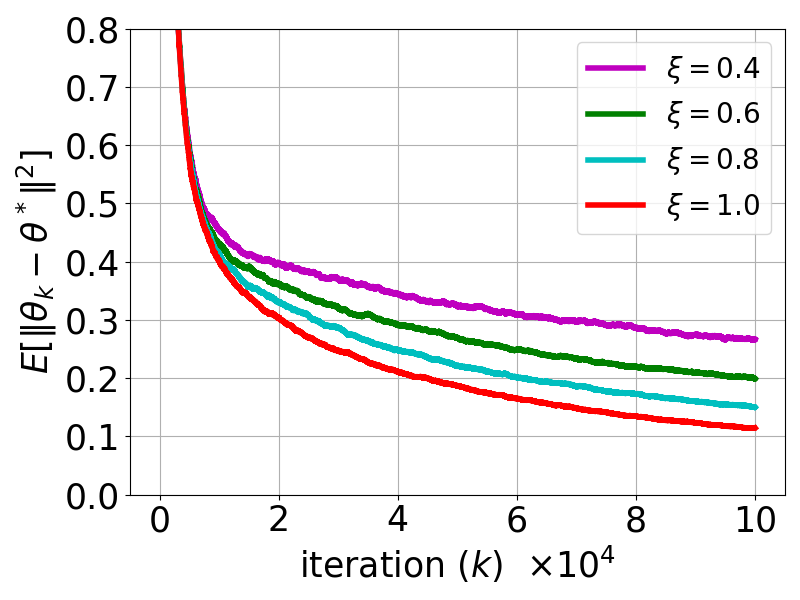}
		\caption{{\small Convergence for diminishing stepsizes}}
		\label{fig:7}
	\end{minipage}%
\hfill
	\begin{minipage}{.40\textwidth}
		\centering
		\includegraphics[width=\linewidth]{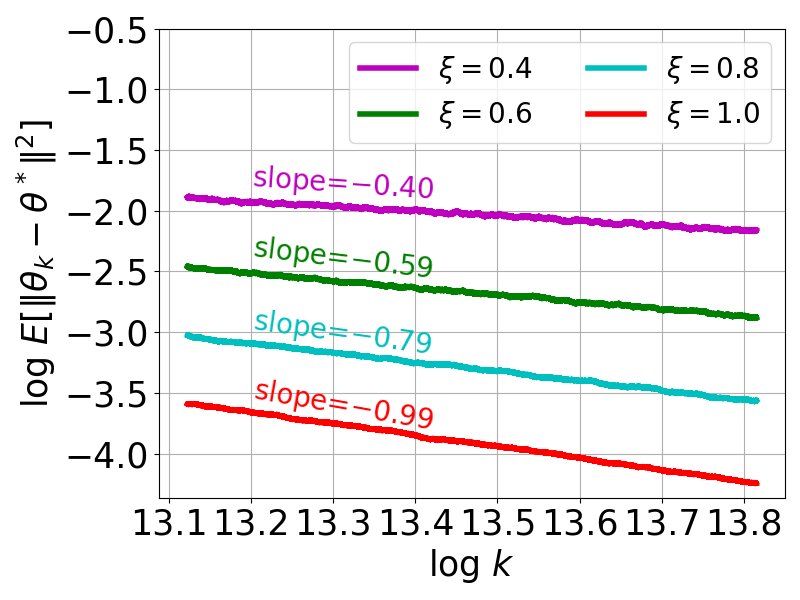}
		\caption{{\small Asymptotic convergence rate}}
		\label{fig:8}
	\end{minipage}
\end{figure}

\end{appendix}

\end{document}